\documentclass{article}

\usepackage[all]{xy}
\usepackage[british]{babel}
\usepackage{amsmath,amscd,amssymb,amsthm,url, bbm}
\usepackage[a4paper,body={16cm,23cm}]{geometry}
\usepackage{color}

\definecolor{dblue}{rgb}{0,0,0.7}
\newtheoremstyle{mythm}{11pt}{11pt}{\it\color{dblue}}{}{\bf\color{dblue}}{.}{\newline}{}
\theoremstyle{mythm}
\newtheorem{thm}{Theorem}
\newtheorem{prop}[thm]{Proposition}
\newtheorem{cor}[thm]{Corollary}
\newtheorem{lem}[thm]{Lemma}

\newtheorem*{fullparityconj}{Full Parity Conjecture}
\newtheorem*{ellparityconj}{$\ell$-Parity Conjecture}

\input cyracc.def 
\font\tencyr=wncyr10 \def\russe{\tencyr\cyracc} 
\font\sevencyr=wncyr7 \def\russesmall{\sevencyr\cyracc} 
\def\Sha{\text{\russe{Sh}}} 
\def\smallSha{\text{\russesmall{Sh}}}

\DeclareMathOperator{\divi}{div}
\DeclareMathOperator{\coker}{coker}
\DeclareMathOperator{\rank}{rank}
\DeclareMathOperator{\corank}{corank}
\DeclareMathOperator{\ord}{ord}
\DeclareMathOperator{\Gal}{Gal}
\DeclareMathOperator{\Sel}{Sel}
\DeclareMathOperator{\Aut}{Aut}

\DeclareMathOperator{\fl}{fl}

\DeclareMathOperator{\id}{id}
\DeclareMathOperator{\cond}{cond}
\DeclareMathOperator{\Spec}{Spec}

\newcommand{\FF}{\mathbb{F}}
\newcommand{\QQ}{\mathbb{Q}}
\newcommand{\ZZ}{\mathbb{Z}}

\newcommand{\cyclic}[1]{{}^{\ZZ}\!/\!{}_{#1 \ZZ}}
\newcommand{\sss}{\scriptscriptstyle}
\newcommand{\qE}{q_{\sss E}}
\newcommand{\eLK}{e_{\sss L/K}}
\newcommand{\zV}{z_{\sss V}}
\newcommand{\Hfl}{H_{\fl}}
\newcommand{\Hflc}{H_{\fl,c}}
\newcommand{\phiSha}{\phi_{\smallSha}}
\newcommand{\hatphiSha}{\hat{\phi}_{\smallSha}}
\newcommand{\oomega}{\underline{\omega}}

\newdir{o>}{{\mkern-10mu\succ}} 
\newdir{>o}{{\succ\mkern-5mu}}  
\newdir{<o}{{\prec\mkern-10mu}}

\begin{document}

\title{Parity conjectures for elliptic curves over global fields of positive characteristic}

\author{Fabien Trihan and Christian Wuthrich}


\maketitle

\begin{abstract}
 We prove the $p$-parity conjecture for elliptic curves over global fields of characteristic $p > 3$. We also present partial results on the $\ell$-parity conjecture for primes $\ell\neq p$.
\end{abstract}

%
%

\section{Introduction}

Let $K$ be a global field and let $E$ be an elliptic curve defined over $K$. The conjecture of Birch and Swinnerton-Dyer asserts that the rank of the Mordell-Weil group $E(K)$ is equal to the order of vanishing of the Hasse-Weil $L$-function $L(E/K,s)$ as $s=1$. A weaker question is to know whether these two integers have at least the same parity. This seems more approachable because the parity of the order of vanishing on the analytic side can by expressed in more algebraic terms through local root numbers -- at least when the $L$-function is known to have an analytic continuation. Let $w(E/K)\in \{\pm 1\}$ be the global root number of $E$ over $K$ which is equal to the product of local root numbers $\prod_v w(E/K_v)$ as $v$ runs over all places in $K$. The local terms $w(E/K_v)$ were defined by Deligne without reference to the $L$-function, see~\cite{rohrlich_wd} for the definition.
So we can formulate the following conjecture.
\begin{fullparityconj}
 We have $(-1)^{\rank E(K)}  = w(E/K)$.
\end{fullparityconj}
This conjecture is unproven except for specific cases. We will focus on the following easier question. Let $\Sha(E/K)$ be the Tate-Shafarevich group defined as the kernel of the localisation maps
$H^1(K,E) \to \prod_{v} H^1(K_v,E)$ in Galois cohomology.
For a prime $\ell$, the $\ell$-primary Selmer group $\Sel_{\ell^{\infty}}(E/K)$ fits into an exact sequence
\begin{equation}\label{mwselsha_eq}
 0 \to E(K)\otimes {}^{\QQ_{\ell}}\!/\!{}_{\ZZ_{\ell}} \to \Sel_{\ell^{\infty}}(E/K) \to \Sha(E/K)[\ell^{\infty}] \to 0.
\end{equation}
If the characteristic of $K$ is prime to $\ell$, we may define it as the preimage of $\Sha(E/K)[\ell^{\infty}]$ under the map $H^1(K,E[\ell^\infty])\to H^1(K,E)[\ell^{\infty}]$. If the characteristic is equal to $\ell$, then one should use flat instead of Galois cohomology, see section~\ref{selmer_sec} for definitions. The theorem of Mordell-Weil shows that the dual of $\Sel_{\ell^{\infty}}(E/K)$ is a finitely generated $\ZZ_{\ell}$-module, whose rank we will denote by $r_{\ell}$. In particular, \eqref{mwselsha_eq} is a short exact sequence of finite cotype $\ZZ_l$-modules for any prime $\ell$. Since it is conjectured that $\rank E(K) = r_{\ell}$, we can make the following conjecture, which seems more approachable as it links two algebraically defined terms.
\begin{ellparityconj}
 Let $\ell$ be a prime.
 We have $(-1)^{r_{\ell}} = w(E/K)$.
\end{ellparityconj}

These conjectures have attracted much attention in recent years and the $\ell$-parity conjecture is now known in many cases, in particular when the ground field is $K=\QQ$ by work of the Dokchitser brothers~\cite{dok_isogeny, dok_nonab, dok_reg, dok_modsquares, dok_09},  Kim~\cite{kim}, Mazur and Rubin~\cite{mazur_rubin}, Nekov\'{a}\v{r}~\cite{nekovar2,nekovar3,nekovar4}, Coates, Fukaya, Kato, and Sujatha~\cite{cfks} and others.

In this article, we restrict our attention to the case of positive characteristic. So, we suppose from now on that $K$ is a global field of characteristic $p>3$ with constant field $\FF_q$. 
The main result of this article is the following theorem.
\begin{thm}\label{pparity_thm}
 The $p$-parity conjecture holds for any elliptic curve $E$ over a global field $K$ of characteristic $p>3$.
\end{thm}
The proof consists of two steps: first a local calculation linking the local root number to local data on the Frobenius isogeny on $E$, carried out in section~\ref{local_sec}; followed by the use of global duality in section~\ref{duality_sec}. Luckily, we do not have to treat all individual cases of bad reduction for the local considerations, since we are able to use a theorem of Ulmer~\cite{ulmer_gnv} to reduce to the semistable case. This is done in section~\ref{red_ss_sec}.

The proof follows closely the arguments in~\cite{dok_isogeny} and Fisher's appendix of~\cite{dok_nonab}. We repeat it here in details, both for completeness and to make the reader aware of a few subtleties; for instance, it is to note that the Frobenius isogeny $F$ and its dual $V$ do not play an interchangeable role.

The hardest part concerns the global duality. The relevant dualities that we need for our conclusion have never appeared in the literature and we are forced to prove them. We think that it is worthwhile to include in section~\ref{pparity_sec} a general formula for the parity of the corank of the $p$-primary Selmer group and a local formula for the root number in section~\ref{rootno_sec}.

Originally, global dualities have appeared in Cassels' work~\cite{cassels} on the invariance under isogenies of the conjecture of Birch and Swinnerton-Dyer.  It should be noted that one could use our duality statements to prove this in the case of characteristic $p>0$, but there is no need for this. In fact it is known by~\cite{kato_trihan} that the conjecture of Birch and Swinnerton-Dyer is equivalent to the finiteness of the Tate-Shafarevich group -- and it is clear that the latter question is invariant under isogeny.

The second main result of this paper is a proof of the $\ell$-parity conjecture when $\ell\neq p$ in some cases. Write $\mu_\ell$ for the $\ell$-th roots of unity.

\begin{thm}\label{ellparity_thm}
 Let $E/K$ be an elliptic curve and let $\ell> 2$ be a prime different from $p$.
 Furthermore assume that 
 \begin{enumerate}
  \item $a=[K(\mu_{\ell}):K]$ is even, and 
  \item the analytic rank of $E$ does not grow by more than $1$ in the constant quadratic extension $K\cdot \FF_{q^2}/K$.
 \end{enumerate}
 Then the $\ell$-parity conjecture holds for $E/K$.
\end{thm}
The proof will be presented in section~\ref{ell_sec}. Its main ingredients are the non-vanishing results of Ulmer in~\cite{ulmer_gnv} and the techniques for proving the parity conjectures from representation theoretic considerations as explained in~\cite{dok_modsquares, dok_sd}.

Although the conditions will be fulfilled for many curves, the methods in this paper fail to give a complete proof of the $\ell$-parity conjecture. See the remarks at the beginning of section~\ref{ell_sec} and the more detailed section~\ref{fail_sec} for an explanation of why we are not able to extend the proof any further.

\subsection{Notations}
The constant field of the global field $K$ of characteristic $p>3$ is the finite field $\FF_q$ for some power $q$ of $p$.
Let $C$ be a smooth, geometrically connected, projective curve over $\FF_q$ with function field $K$. Let $E/K$ be an elliptic curve, which we will assume to be non-isotrivial (i.e. the $j$-invariant of $E$ is transcendental over $\FF_q$). We fix a Weierstrass equation
\begin{equation}\label{w_eq}
 E \colon\quad y^2\, =\, x^3\, +\, A\,x\, +\, B
\end{equation}
with $A$ and $B$ in $K$ and the corresponding invariant differential $\omega = \tfrac{dx}{2y}$. By $F\colon E \to E'$ we denote the Frobenius isogeny of degree $p$ whose dual $V\colon E'\to E$ is the Verschiebung.

If $f\colon A\to B$ is a homomorphism of abelian groups, we write
\begin{equation*}
 z(f) = \frac{\# \coker(f)}{\# \ker(f)}
\end{equation*}
provided that the kernel and the cokernel of $f$ are finite. For any abelian group (or group scheme) $A$ and integer $m$, we denote by $A[m]$ the $m$-torsion part of it; and, for any prime $\ell$, the $\ell$-primary part will be denoted by $A[\ell^{\infty}]$. 

The Pontryagin dual of an abelian group $A$ is written $A^{\vee}$. If the Pontryagin dual of $A$ is a finitely generated $\ZZ_{\ell}$-module for some prime $\ell$, then we write $\divi(A)$ for its maximal divisible subgroup and let $A_{\divi}$ denote the quotient of $A$ by $\divi(A)$.

%

\section{Reduction to the semistable case}\label{red_ss_sec}

Before, we start we should mention that the conjecture of Birch and Swinnerton-Dyer is known for isotrivial curve $E$ by the work of Milne~\cite{milne_isotrivial}. So for the rest of the paper we will assume that $E$ is not isotrivial as otherwise the parity conjectures are known. In particular, it follows from this assumption that $E/K$ is ordinary. The parity conjecture is also known in the following cases:

\begin{prop}\label{an_rk0&1}
Let $A/K$ be an abelian variety over a function field of characteristic $p>0$ and let $\ell$ be a prime ($\ell=p$ is allowed). The analytic rank of $A/K$ is always greater or equal to the $\ell$-corank of the Selmer group. If the analytic rank of $A/K$ is zero, then the conjecture of Birch and Swinnerton-Dyer holds. If the analytic rank is $1$ then it coincides with the $\ZZ_\ell$-corank of the $\ell$-primary Selmer group.
\end{prop}

Note that if we restricted ourselves to elliptic curves and to the case $\ell \neq p$, then this result could already be deduced from the work of Artin and Tate~\cite{tate}.

\begin{proof} By~3.2 in~\cite{kato_trihan}, the Hasse-Weil $L$-function of $A/K$ can be expressed as an alternating product of characteristic polynomials of some operators $\phi^i_\ell$ acting on a finite dimensional $\QQ_\ell$-vector space $H^i_{\QQ_\ell}$, with $i=0,1,2$. Then by~3.5.1 in~\cite{kato_trihan}, the order at $s=1$ of the Hasse-Weil $L$-function can be interpreted as the multiplicity of the eigenvalue 1 for
the operator $\phi^1_\ell$ on $H^1_{\QQ_\ell}$. Following the notations of~3.5 in~\cite{kato_trihan}, let $I_{3,\ell}$ denote the part of $H^1_{\QQ_\ell}$ on which the operator $\id-\phi^1_\ell$ acts nilpotently and let $I_{2,\ell}$ denote the kernel of $\id-\phi^1_\ell$, such that we have the inclusions:
$$I_{2,\ell}\subset I_{3,\ell}\subset H^1_{\QQ_\ell}.$$
Since by~3.5.1 in~\cite{kato_trihan}, the operator $\id-\phi^i_\ell$ is an isomorphism for $i=0,2$, it follows that the analytic rank of $A/K$ is equal to the dimension of $I_{3,\ell}$. On the other hand, it follows from~3.5.5 and~3.5.6 in~\cite{kato_trihan}, that the $\ell$-corank of the Selmer group of $A/K$ is the dimension of $I_{2,\ell}$ so that we deduce that the analytic rank of $A/K$ is always greater or equal to the $\ell$-corank of the Selmer group of $A/K$. If the analytic rank of $A/K$ is trivial, so is the dimension of $I_{3,\ell}$. It implies that the dimension of $I_{2,\ell}$ is zero and by~3.5.6 in~\cite{kato_trihan}, we conclude that the Mordell-Weil group is also of rank zero. We then conclude the proof of the assertion thanks to the main result~1.8 of~\cite{kato_trihan}. If the analytic rank of $A/K$ is one, then $\phi^1_\ell$ acts like the identity on $I_{3,\ell}$ and therefore $I_{2,\ell}=I_{3,\ell}$ and the second assertion immediately follows.
\end{proof}

The following proposition will be used at several places to reduce the conjecture to easier situations.

\begin{prop}\label{red_prop}
 Let $E/K$ be a non-isotrivial curve and $L/K$ a separable extension. Let $\ell$ be a prime. Assume one of the following three
 conditions:
 \begin{enumerate}
  \item\label{odd_cond} $\ell\neq p$ and the extension $L/K$ is a Galois extension of odd degree.
  \item\label{zero_cond} The analytic rank of $E$ does not grow in $L/K$.
  \item\label{one_cond} $\ell\neq p$ and the analytic rank of $E$ does not grow by more than $1$ in $L/K$.
 \end{enumerate}
 Then the $\ell$-parity conjecture for $E/K$ holds if and only if the $\ell$-parity conjecture for $E/L$ is known.
\end{prop}
\begin{proof}
 If condition~(\ref{odd_cond}) holds then the conclusion follows directly from Theorem~1.3 in~\cite{dok_sd}. Note already here that the complete paper~\cite{dok_sd} and its proofs hold in our situation as long as $\ell\neq p$.

 Suppose now as in condition~(\ref{zero_cond}) that the analytic rank does not grow in $L/K$. Denote by $A/K$ the Weil restriction of $E$ under $L/K$ and by $B/K$ the quotient of $A$ by the natural image of $E$ in it. Since 
\begin{equation*}
 L(E/L,s) = L(A/K,s) = L(E/K,s)\cdot L(B/K,s)
\end{equation*}
we see that that the analytic rank of $B/K$ is zero and therefore by Proposition~\ref{an_rk0&1}, the full Birch and Swinnerton-Dyer conjecture holds. In particular, the Mordell-Weil rank of $B/K$ is zero and its Selmer group is a finite group.  Moreover, we have an exact sequence
\begin{equation}\label{sels_ses}
 \Sel_{\ell^{\infty}}(E/K) \to \Sel_{\ell^{\infty}}(A/K) \to \Sel_{\ell^{\infty}}(B/K),
\end{equation}
and the kernel of the first map lies in $B(K)[\ell^{\infty}]$, which is a finite group. Hence we conclude that
$r_{\ell}$ is equal to the corank of $\Sel_{\ell^{\infty}}(A/K)$ and, by Proposition~3.1 in~\cite{mazur_rubin}, this is the same as the corank of $\Sel_{\ell^{\infty}}(E/L)$. So we are able to deduce the $\ell$-parity for $E/K$ from the $\ell$-parity for $E/L$.

 Finally, suppose that $\ell\neq p$ and that the analytic rank grows exactly by $1$; so we are under condition~(\ref{one_cond}). Then  we know by Proposition~\ref{an_rk0&1} that the rank of $\Sel_{\ell^{\infty}}(B/K)$ is less or equal to $1$. We wish to exclude the possibility that it is $0$, so assume by now that $\Sel_{\ell^{\infty}}(B/K)$ is finite. But this means that $\Sha(B/K)[\ell^{\infty}]$ is finite and hence the full conjecture of Birch and Swinnerton-Dyer holds by~\cite{kato_trihan} again. So we reach a contradiction since we would have $0 = \rank B(K) = \ord_{s=1} L(B/K,s) = 1$. Hence we have shown that the corank of $\Sel_{\ell^{\infty}}(B/K)$ is $1$.

Note that the left-hand map in~\eqref{sels_ses} still has finite kernel. We will show now that right-hand map has finite cokernel, too. Let $\Sigma$ be the finite set of places in $K$ of bad reduction for $E$. Write $G_{\Sigma}(K)$ for the Galois group of the maximal separable extension of $K$ which is unramified outside $\Sigma$. Note that from the  
definition of the Selmer group, we find the following diagram with exact rows and columns
\begin{equation*}
\xymatrix@C-8pt{
 0\ar[d] & 0\ar[d] & \\ 
 \Sel_{\ell^{\infty}}(A/K)\ar[d]\ar[r] &
 \Sel_{\ell^{\infty}}(B/K) \ar[d]& \\
 H^1\bigl(G_{\Sigma}(L),A[\ell^{\infty}]\bigr)\ar[d]\ar[r] & 
 H^1\bigl(G_{\Sigma}(K),B[\ell^{\infty}]\bigr)\ar[d]\ar[r] &
 H^2\bigl(G_{\Sigma}(K),E[\ell^{\infty}]\bigr)\ar[r]^{r} & 
 H^2\bigl(G_{\Sigma}(K),A[\ell^{\infty}]\bigr)  \\
 \bigoplus_{v\in \Sigma} H^1(K_v,A)[\ell^{\infty}] \ar[r] &
 \bigoplus_{v \in \Sigma} H^1(K_v,B)[\ell^{\infty}]
 }
\end{equation*}
We know that the bottom groups are finite as they are dual to $\varprojlim A(K_v)/\ell^k$ and $\varprojlim B(K_v)/\ell^k$ respectively. Hence we see from the snake lemma that we only have to prove that the kernel of $r$ is finite. Shapiro's lemma shows that $H^2(G_{\Sigma}(K), A[\ell^\infty])$ is isomorphic to $H^2(G_{\Sigma}(L),E[\ell^\infty])$ and hence the map $r$ is simply the restriction map. As its kernel will only get larger when increasing $L$, we may assume that $L/K$ is Galois. Then the kernel of the restriction is contained in the part of $H^2(G_{\Sigma}(K),E[\ell^{\infty}])$ that is killed by $[L:K]$. Hence it is finite, because  $H^2(G_{\Sigma}(K),E[\ell^{\infty}])$ is a discrete abelian group with finite $\ZZ_{\ell}$-corank.

Therefore, we conclude again that the corank of the $\ell$-primary Selmer group increased by exactly $1$ in $L/K$.
\end{proof}
\begin{cor}\label{toss_cor}
 If the $\ell$-parity conjecture holds for all semistable elliptic curves, then it holds for all elliptic curves.
\end{cor}
\begin{proof}
 Theorem~11.1 in~\cite{ulmer_gnv} proves that there is a separable extension $L/K$ such that the reductions of $E$ becomes semistable and the analytic rank does not grow in $L/K$.
 \end{proof}

The same argument also reduces the full parity conjecture to the semistable case.

%

\section{Local computations}\label{local_sec}

The following computations are purely local and we change the notations for this section. Let $K$ be a local field of characteristic $p>3$ with residue field $\FF_q$. The ring of integers is written $\mathcal{O}$, the maximal ideal $\mathfrak{m}$ and the normalised valuation by $v$. The elliptic curve $E/K$ is given by the equation~\eqref{w_eq}. By changing the equation, if necessary, we may suppose for this section that $A$ and $B$ are in $\mathcal{O}$.

Define $L$ to be the minimal extension of $K$ such that $E'(L)[p] =\cyclic{p}$, or equivalently that $E[F]$ is isomorphic to $\mu[p]$ as a group scheme over $L$. There is a representation 
\begin{equation*}
\rho\colon \Gal(L/K)\to \Aut(E'(L)[p]) \cong (\cyclic{p})^{\times} 
\end{equation*}
which shows that $[L:K]$ divides $p-1$. Define $\bigl(\frac{-1}{L/K}\bigr)\in\{\pm 1\}$ to be the image of $-1$ under the composition of the reciprocity map and $\rho$
\begin{equation*}
 K^{\times} \to \Gal(L/K)\to (\cyclic{p})^{\times}.
\end{equation*}
So $\bigl(\frac{-1}{L/K}\bigr) = +1$ if and only if $-1$ is a norm from $L^{\times}$ to $K^{\times}$.

We will also consider $\zV = z\bigl(V\colon E'(K) \to E(K)\bigr)$, which is a certain power of $p$. Put
\begin{equation*}
 \sigma = \sigma(E/K) = \begin{cases} +1 \quad&\text{ if $z_V$ is a square and} \\ -1 &\text{ otherwise.} \end{cases}
\end{equation*}
It is important to note that we cannot define $z\bigl(F\colon E(K) \to E'(K)\bigr)$ since its cokernel will never be finite.

Finally, as in the introduction, we let $w = w(E/K)$ be the local root number of $E$ over $K$, as defined by Deligne and well explained in~\cite{rohrlich_wd}.
The aim of this section is to show the following theorem.
\begin{thm}\label{local_thm}
 Let $K$ be a local field of characteristic $p>3$.
 For any non-isotrivial elliptic curve $E/K$ whose reduction is not additive and potentially good, we have $w(E/K) = \bigl(\frac{-1}{L/K}\bigr) \cdot \sigma(E/K)$.
\end{thm}
We will prove this theorem by treating each type of reduction separately. In the last section of this paper, we will prove this local theorem without the assumption on the reduction using global methods.
See Conjecture~5.3 in~\cite{dok_09} for the analogue in characteristic zero. In particular, the following computations show that the analogy should take places above $p$ in characteristic zero to supersingular places in characteristic $p$.

Recall the definition of the Hasse invariant $\alpha = A(E,\omega)$ associated to the given integral equation~\eqref{w_eq}. 
Write $\mathcal{F}$ for the formal group of $E$ over $\mathcal{O}$, and similarly $\mathcal{F}'$ for the formal group for the isogenous curve $E'$ given by the 
integral equation
\begin{equation*}
 E'\colon \quad y'^2 \, = \, x'^3\, + \,A^p\, x'\,+\,B^p.
\end{equation*}

Choose $t=-\frac{x'}{y'}$ as the parameter of the formal group $\mathcal{F}'$. 
Then the formal isogeny $V_1$ of the Verschiebung $V$ is of the form
\begin{equation*}
 \xymatrix@R=3mm{
 V_1 \colon \mathcal{F}'(\mathfrak{m})\ar@{-o>}[r] &  \mathcal{F}(\mathfrak{m}) \\
 t  \ar@{|-o>}[r] &  \alpha \cdot G(t) + H(t^p)
 }
\end{equation*}
 for some $G(t) = t+\cdots \in \mathcal{O}[\![t]\!]$ and $H(t) = u\cdot t +\cdots \in \mathcal{O}[\![t]\!]$ with $u$ in $\mathcal{O}^{\times}$. See section~12.4 in~\cite{katz_mazur} for other descriptions of the Hasse invariant $\alpha$. 

We begin now the proof of Theorem~\ref{local_thm}. For the computation of the local root number $w$, we can simply refer to Proposition~19 in~\cite{rohrlich_wd}, where we find that $w=-1$ if the reduction is split multiplicative and $w=+1$ if it is good or non-split multiplicative.

\subsection{Good reduction}

\begin{prop}\label{good_prop}
 Suppose $E/K$ has good reduction. Then $w=+1$. The quantities 
 $\sigma$ and $\bigl(\frac{-1}{L/K}\bigr)$ are $+1$ if and only if $q^{v(\alpha)}$ is a square.
 In particular, if the reduction is ordinary then $\sigma=\bigl(\frac{-1}{L/K}\bigr)=+1$.
\end{prop}
\begin{proof}
 We may suppose that the equation~\eqref{w_eq} is minimal, i.e. that it has good reduction. We then have the diagram
\begin{equation*}
 \xymatrix@C=10mm{
 0 \ar@{-o>}[r] & \mathcal{F}'(\mathfrak{m})  \ar@{-o>}[r] \ar[d]_{V_1} & E'(K) \ar@{-o>}[r] \ar[d]_{V} & \tilde{E'}(\FF_q) \ar@{-o>}[r]\ar[d] & 0 \\
 0 \ar@{-o>}[r] &  \mathcal{F}(\mathfrak{m})  \ar@{-o>}[r] & E(K) \ar@{-o>}[r] & \tilde{E}(\FF_q) \ar@{-o>}[r] & 0
 }
\end{equation*} 
where $\tilde{E}$ denotes the reduction of $E$. The isogenous curves $\tilde{E}$ and $\tilde{E'}$ over $\FF_q$ have the same number of points, so the kernel and cokernel of this map have the same size. Hence $\zV = z(V_1)$. 

For any $N\geqslant 1$,
\begin{equation*}
 \frac{\mathcal{F}(\mathfrak{m}^N)}{\mathcal{F}(\mathfrak{m}^{N+1})} \cong \frac{\mathfrak{m}^N}{\mathfrak{m}^{N+1}} \cong \frac{\mathcal{F}'(\mathfrak{m}^N)}{\mathcal{F}'(\mathfrak{m}^{N+1})}
\end{equation*}
and so the same argument shows that $\zV = z(V_1) = z\bigl(V_N\colon \mathcal{F}'(\mathfrak{m}^{N}) \to \mathcal{F}(\mathfrak{m}^{N}) \bigr)$. 

We claim that if $N>\upsilon(\alpha)$ then $V_N$ maps $\mathcal{F}'(\mathfrak{m}^N)$ bijectively onto $\mathcal{F}(\mathfrak{m}^{N+\upsilon(\alpha)})$. If $t$ has valuation at least $N$, then the valuation of $\alpha\,t$ is smaller than the valuation of $u\cdot t^p$. Therefore $\upsilon(V_N(t)) = \upsilon(\alpha) + \upsilon(t)$. This shows that $V_N$ maps $\mathcal{F}'(\mathfrak{m}^N)$ injectively to $\mathcal{F}(\mathfrak{m}^{N+\upsilon(\alpha)})$. In particular, the kernel $\ker(V_N)$ is trivial. 

Let $s$ have valuation $\upsilon(s) \geqslant N+\upsilon(\alpha)$. Put $t_0 = s/\alpha$. Then $t_0$ is close to a zero of $g(t) = V_N(t) - s$. Namely $g(t_0) = \alpha \,a\, t_0^2 +\cdots +u\, t_0^p+\cdots$ has valuation at least $2\,\upsilon(s)-\upsilon(\alpha) \geqslant 2\,N+\upsilon(\alpha) > 2\,\upsilon(\alpha)$, if we write $G(t) = t+a\,t^2+\cdots$ for some $a\in\mathcal{O}$. Since $g'(t_0) = \alpha + 2\,\alpha\, a \, t_0 + \cdots$ has valuation $\upsilon(\alpha)$, Hensel's lemma shows that there is a $t$ close to $t_0$ such that $g(t) = 0$, i.e. such that $V_N(t) = s$.

We conclude that the cokernel of $V_N$ is equal to the index of $\mathcal{F}(\mathfrak{m}^{N+\upsilon(\alpha)})$ in $\mathcal{F}(\mathfrak{m}^N)$. Hence $\zV = z(V_N) = q^{\upsilon(\alpha)}$. In particular $\zV=1$ if the reduction is ordinary, i.e. when $\alpha$ is a unit in $\mathcal{O}$.

Let $\eLK$ be the ramification index of $L/K$. If the reduction is good ordinary, then the inertia group acts trivially on $T_p E'$, which is a $\ZZ_p$-module of rank 1. Hence $L/K$ is unramified and we have immediately that $\bigl(\frac{-1}{L/K}\bigr)=+1$.
\begin{lem}
The parity of  ${v(\alpha)}$ is equal to the parity of $\frac{p-1}{\eLK}$.
 \end{lem}
\begin{proof}
If $E$ has good ordinary reduction, then ${v(\alpha)}=0$, $e_{L/K}=1$ and $p-1$ is even so that the assertion is true.  If $E$ has good supersingular reduction, then since $E'(L)[p]$ contains a non-trivial point $P=(x'_P,y'_P)$, but the reduction does not contain a point of order $p$, there exist a $t_P = -x'_P/y'_P$ in the maximal ideal $\mathfrak{m}_L$ of $L$ such that $V_1(t_P) = 0$. From  $V_1 (t_P) = \alpha\,t_P +\cdots + u\, t_P^p +\cdots$, we see that the valuation of $\alpha t_P$ and $ut_P^p$ must cancel out. Hence $\upsilon_L(\alpha) = \eLK \cdot\upsilon(\alpha) = (p-1)\cdot \upsilon_L(t_P)$, where $\upsilon_L$ denotes the normalised valuation in $L$. So if the valuation of $t_P$ is odd, we have proved the assertion. 

Assume that $\upsilon_L(t_P)$ is even. Then $\upsilon(\alpha)$ is also even and we have to show that $\frac{p-1}{\eLK}$ is even. The extension $L/K(x'_P)$ is generated by $t_P$ whose square belongs to $K(x'_P)$; so this extension is either unramified quadratic or trivial. If $L=K(x'_P)$, then $\Gal(L/K)$ acts on the set of $\{x'_P\vert O\neq P\in E'(L)[p]\}$ and hence $[L:K]$ divides $\frac{p-1}{2}$, so $\frac{p-1}{[L:K]}$ is even. Otherwise, if $L$ is an unramified quadratic extension of $K(x'_P)$, then $\eLK = e_{K(x'_p)/K}$ and $p-1$ is divisible by $[L:K] = 2 \eLK f_{K(x'_p)/K}$. So $\frac{p-1}{\eLK}$ is even.
\end{proof}

Now we can conclude the proof of Proposition~\ref{good_prop}. Lemma~12 in~\cite{dok_isogeny}, whose proof is valid even if the characteristic of $K$ is not zero, says that
$\bigl(\frac{-1}{L/K}\bigr)=+1$ if and only if $q$ is a square or if $\frac{p-1}{e_{L/K}}$ is even. The previous lemma suffices now to conclude.
\end{proof}

 In the good supersingular case, $L/K$ may or may not be totally ramified. We illustrate this with two examples.

 We take $p=5$, $w>0$ any integer, and the curve $E$ given by the minimal Weierstrass equation
\begin{equation*}
 y^2\ = \ x^3 \ + \ T^w \cdot x \ + \ 1
\end{equation*}
over $\mathcal{O}=\FF_5[\![T]\!]$. The Hasse invariant is $\alpha = 2\cdot T^w$. The reduction is good, but supersingular. The division polynomial $f_V$ associated to the isogeny $V$ can be computed to be equal to
\begin{equation*}
 f_V(x) = 2\, T^w\,x^2+4\,T^{2w}\,x + (4+3\,T^{3w}+T^{6w})\,.
\end{equation*}

First we take the case $w=2\,m$ is even. Then we can make the substitution $X = T^m \cdot x$ to get
\begin{equation*}
 f_V(x) = 2\, X^2+4\,T^{3m}\,X + (4+3\,T^{6m}+T^{12m})\,.
\end{equation*}
We see that $K(x_P)$ is a quadratic unramified extension of $K$. The quantity $\frac{p-1}{\eLK}$ will certainly be even.

Now, take $w=2\,m-1$ to be odd with $m>1$. This time the substitution $X=T^m\cdot x$ gets us to
\begin{equation*}
 T\cdot f_V(x) = 2\, X^2+4\,T^{3m-2}\,X + T\cdot (4+3\,T^{6m-3}+T^{12m-6})\,.
\end{equation*}
Therefore $K(x_P)/K$ will be a ramified extension of degree $2$. The valuation of $x_P$ over $K(x_P)$ is odd, so we have to make a further extension $L/K(x_P)$, again ramified of degree $2$, to have a $p$-torsion point in $E'(L)$. So $\eLK =4$ and $\frac{p-1}{\eLK}$ is odd.

\subsection{Split multiplicative}

\begin{prop}\label{split_prop}
 Suppose $E/K$ has split multiplicative reduction. Then $w(E/K) = -1$, $\bigl(\frac{-1}{L/K}\bigr) = +1$, and $\sigma(E/K) = -1$.
\end{prop}
\begin{proof}
Let $\qE\in K^{\times}$ be the parameter of the Tate curve which is isomorphic to $E$ over $K$. Then the isogenous curve $E'$ has parameter ${\qE}^p$ and the Frobenius map
\begin{equation*}
 \xymatrix{V\colon \frac{ K^{\times} }{ ({\qE}^p)^{\ZZ} } \ar@{-o>}[r] &  \frac{ K^{\times} }{ (\qE)^{\ZZ} }   }
\end{equation*}
is induced by the identity on $K^{\times}$. Hence $V$ has a kernel with $p$ elements and is surjective, so $z_V = \frac{1}{p}$ and $\sigma = -1$.

Since $E'$ has already a $p$-torsion point over $K$, we have $L=K$ and $\bigl(\frac{-1}{L/K}\bigr)=+1$.
\end{proof}

\subsection{Non-split multiplicative}
\begin{prop}\label{nonsplit_prop}
Suppose $E/K$ has non-split multiplicative reduction. Then $$w(E/K) =\bigl(\frac{-1}{L/K}\bigr)=\sigma(E/K) = +1.$$
\end{prop}
\begin{proof}
 There is a quadratic extension $K'$ over which $E$ has split multiplicative reduction. So either $L=K$ or $L=K'$. Let $E^{\dagger}$ be the quadratic twist of $E$ over $K'$. Up to $2$-torsion groups, we have $E(L)= E(K)\oplus E^{\dagger}(K)$. Since $E^{\dagger}$ has split multiplicative reduction over $K$ there is a $p$-torsion point in $E^{\dagger}(K)$. So $L=K'$.

From the previous section, we know that $\zV$ for $E/L$ and $E^{\dagger}/K$ both are equal to $\frac{1}{p}$. So by the above formula for $E(L)$ up to $2$-torsion, we get that $\zV$ for $E/K$ is $1$. So $\sigma = +1$. Since $L/K$ is unramified, $ \bigl(\frac{-1}{L/K}\bigr) = +1$. 
\end{proof}

\subsection{Additive potentially multiplicative}
\begin{prop}
 Suppose $E/K$ has additive, potentially multiplicative reduction.
 Let $\chi\colon K^{\times} \to \{\pm 1\}$ be the character associated to the quadratic ramified extension over which $E$ has split multiplicative reduction.
 Then $w(E/K) =\bigl(\frac{-1}{L/K}\bigr)=\chi(-1)$ and $\sigma(E/K) = +1$.
\end{prop}
\begin{proof}
 The root number is computed by Rohrlich~\cite[19.ii]{rohrlich_wd}. The proof that $\sigma = +1$ is the same as in the non-split multiplicative case. The formula $ \bigl(\frac{-1}{L/K}\bigr) = \chi(-1)$ is clear, too. 
\end{proof}

%

\section{Selmer groups}\label{selmer_sec}

We return to the global situation and we wish to define properly the Selmer groups involved in $p$-descent in characteristic $p$ using flat cohomology.

From now on, $K$ is again a global field with field of constants $\FF_q$ and $E/K$ is a semistable, non-isotrivial elliptic curve.
We denote by $\mathcal{E}$ the N\'eron model of $E/K$ over $C$ and $\mathcal{E}^0$ its connected component containing the identity.
Let $U$ be a dense open subset of $C$ such that $\mathcal{E}$ has good reduction on $U$. The group schemes $\mathcal{E}$ and $\mathcal{E}^0$ coincide over $U$ and we define for any $v\not\in U$ the group of connected components $\Phi_v$ in the fibre above $v$. So we have the following short exact sequence:
\begin{equation*}
0 \to \mathcal{E}^0 \to \mathcal{E} \to \bigoplus_{v\not \in U} \Phi_v \to 0.
\end{equation*} 

Following 2.2 in~\cite{kato_trihan}, the discrete $p^{\infty}$-Selmer group of $E/K$ is defined as
\begin{equation*}
\Sel_{p^\infty}(E/K) := \ker \Bigl[ \Hfl^1 (K,E[p^\infty])\rightarrow
                \prod_{v} \Hfl^1 (K_v,E) \Bigl]
\end{equation*} 
where $E[p^{\infty}]$ is the $p$-divisible group associated to $E$
and $\Hfl$ stands for flat cohomology.
It is known that $\Sel_{p^{\infty}}(E/K)$ fits into the following exact sequence:
\begin{equation}\label{ses1_seq}
  0 \to E(K)\otimes {}^{\QQ_p}\! /\!{}_{\ZZ_p} \to \Sel_{p^\infty}(E/K) \to \Sha(E/K)[p^\infty] \to 0 .
\end{equation} 
This follows from the fact that the Tate-Shafarevich group can also be computed using flat cohomology as
the kernel of  $\Hfl^1 (K,E)\to \prod_v \Hfl^1 (K_v,E)$ since for the elliptic curve $E$ over $K$ or $K_v$, the \'etale and flat cohomology groups coincide (see Theorem~3.9 in~\cite{Milne80}).
Note also that the dual of $\Sel_{p^{\infty}}(E/K)$ is a finitely generated $\ZZ_p$-module by the theorem of Mordell-Weil and the finiteness of $\Sha(E/K)[p]$ (see e.g.~\cite{ulmer_pdescent}).

Let $\phi\colon E\to E'$ be an isogeny of elliptic curves. The map $\phi$ induces a short exact sequence of sheaves in the flat topology
\begin{equation}\label{ses2_seq}
  \xymatrix@1{
  0 \ar[r] & E[\phi] \ar[r] &  E \ar[r]^{\phi} & E' \ar[r] & 0.
  }
\end{equation}
The Selmer group $\Sel_\phi(E/K)$ is defined to be the set of elements in $\Hfl^1(K,E[\phi])$ whose restrictions to $\Hfl^1(K_v,E[\phi])$ lie in the image of the connecting homomorphism $E(K_v)\to \Hfl^1(K_v,E[\phi])$ for all $v$. If $U$ is any open subset of $C$ where $E$ has good reduction, we can also describe $\Sel_\phi(E/K)$ as the kernel of the composed map
\begin{equation*}
  \xymatrix@1{
      \Hfl^1(U,\mathcal{E}[\phi]) \ar[r]  &
     \prod_{v\not\in U} \Hfl^1(K_v,E[\phi]) \ar[r]    &
     \prod_{v\not\in U} \Hfl^1(K_v,E)[\phi].
}
\end{equation*}
Passing to cohomology, the short exact sequence~\eqref{ses2_seq} induces the short exact sequence of finite groups
\begin{equation}\label{selsha_seq}
  \xymatrix@1{
     0 \ar[r] & E'(K)/\phi(E(K)) \ar[r] &  \Sel_\phi (E/K) \ar[r] & \Sha(E/K)[\phi] \ar[r] & 0,
  }
\end{equation} 
where $\Sha(E/K)[\phi]$ is the kernel of the induced map $\phiSha \colon \Sha(E/K) \to \Sha(E'/K)$.

%
%

\section{Global Euler characteristics}\label{duality_sec}

We prove in the next three sections a few results on global dualities for the $p$-primary part of the Tate-Shafarevich group in characteristic $p$ using flat cohomology. The main reference will be~\cite{milne}, but we wish to point the reader to related results in~\cite{GT} and~\cite{Go}. 

Note that most results in these three sections do not need any condition on the reduction. Also, except where mentioned, the characteristic $p$ can be any prime. 

We give a short review of the Oort-Tate classification of finite flat group schemes of order $p$ (see~\cite{oort_tate} for details). Let $X$ be a scheme of characteristic $p>0$. The data of a finite flat group scheme $N$ of order $p$ over $X$ is equivalent to the data of an invertible sheaf $\mathcal{L}$, a section $a\in H^0(C,\mathcal{L}^{\otimes(p-1)})$ and a section $b\in H^0(C,\mathcal{L}^{\otimes(1-p)})$ such that $a\otimes b=0$. We use the notation $N_{\mathcal{L},a,b}$. If $N$ is of height one, then $a=0$. The Cartier dual of $N_{\mathcal{L},a,b}$ is $N_{\mathcal{L}^{-1},b,a}$.

For a scheme $S$ of characteristic $p>0$ and a finite flat group scheme $N/S$ we define the Euler characteristic of $N/S$ as
\begin{equation*}
\chi(S,N) := \prod_i\Bigl(\#\Hfl^i(S,N)\Bigr)^{(-1)^i}
\end{equation*}
whenever the groups $\Hfl^i(S,N)$ are finite.
\begin{lem}\label{globalchi_lem} 
Assume that the prime $p$ is odd. Let $N$ be a finite flat group scheme of order $p$ over $C$. Assume that the Cartier dual $N^D$ of $N$ is of height 1. Then the groups $\Hfl^i(C,N)$ are finite and $\chi(C,N)$ is a square in $\QQ^{\times}$.
\end{lem}
\begin{proof}
The cohomology is finite by Lemma~III.8.9 in~\cite{milne} since $N$ is finite flat.
If $N^D$ has height $1$ then $N$ corresponds to a group scheme $N_{\mathcal{L},a,b}$ with $b=0$.
Now we follow the explanation after problem~III.8.10 in~\cite{milne}.
Since $N$ is the dual of a group scheme of height 1, we have that there is a sequence
\begin{equation*}
  \xymatrix@R-4ex@C+1ex{0\ar[r]& N \ar[r] & \mathcal{L}\ar[r] & \mathcal{L}^{\otimes p}\ar[r] &0,\\
                           &          &  z \ar @{|->} [r] & z^{\otimes p}-a\otimes z & }
\end{equation*}
which is exact by the definition of $N_{\mathcal{L},a,b}$. See also Example~III.5.4 in~\cite{milne}.
Hence we have that $\chi(C,N) = q^{\chi(\mathcal{L})-\chi(\mathcal{L}^{\otimes p})}$. Using Riemann-Roch, we get
\begin{align*}
  \chi(\mathcal{L}) &= \deg(\mathcal{L}) + 1 - g\\
  \chi(\mathcal{L}^{\otimes p}) &= p\cdot \deg(\mathcal{L}) + 1 - g
\end{align*}
and therefore we find the formula
\begin{equation*}
 \chi(C,N) = q^{(p-1)\deg \mathcal{L}}.
\end{equation*}
So the lemma follows from the fact that $p$ is odd.
\end{proof}

For any place $v$ in $K$, we denote by $\vert\cdot\vert_v$ the normalised absolute value of the completion $K_v$. In particular the absolute value of a uniformiser is $q_v^{-1}$ where $q_v$ denotes the number of elements in the residue field.
\begin{lem}\label{localchi_lem} 
Let $N=N_{\mathcal{L},a,b}$ be a finite flat group scheme of order $p>2$ over the ring $O_v$ of integers in $K_v$. Assume that $N_{K_v}$ is \'etale. Then the Euler characteristic of $N$ is well-defined and we have
\begin{equation*}
\chi(O_v,N)\equiv\vert a\vert_v^{-1}
\end{equation*} 
modulo squares in $\QQ^{\times}$.
\end{lem}
\begin{proof} 
The invertible sheaf $\mathcal{L}$ is $c^{-1}\cdot O_v$ for some $c\in K_v^{\times}$. Then by~III.0.9.(c) in~\cite{milne}, we have $N_{\mathcal{L},a,b}\cong N_{O_v,ac^{p-1},bc^{1-p}}$. Using Remark~III.7.6 and the Example after Theorem~III.1.19 on page~244 of~\cite{milne}, we have $\chi(O_v,N)=\vert a\cdot c^{p-1}\vert_v^{-1}\equiv \vert a\vert_v^{-1}$ modulo squares in $\QQ^{\times}$.
\end{proof}

For a scheme $S$ of characteristic $p>0$ and a scheme $X/S$, we denote by $X'$ the fibre product $X\times_S S$ where the map $S\to S$ in this product is the absolute Frobenius of $S$. By the universal property of the fibre product, we have a map $F\colon X\to X'$ called the relative Frobenius. If moreover $X/S$ is a flat group scheme, then there exists a map $V\colon X'\to X$ called the Verschiebung such that $V\circ F$ and $F\circ V$ induce $[p]$, the multiplication by $p$ (see~\cite{SGA3}, VII). In particular, $F\colon E\to E'$ is a $p$-isogeny of elliptic curves which extends to the N\'eron models of $E$ and $E'$ by its universal property. Since the N\'eron model of $E'$ is $\mathcal{E}'$, this map is just the relative Frobenius $F\colon \mathcal{E}\to \mathcal{E}'$.

Over the field $K$, or more generally over any open subset $U$ in $C$ where $E$ has good reduction, Proposition~2.1 of~\cite{ulmer_pdescent} shows that $E[F]=N_{\oomega^{-1},0,\alpha}$ and $E[V]=N_{\oomega,\alpha,0}$, where $\alpha$ is the Hasse invariant of $E$ and where $\oomega$ is the invertible sheaf $\pi_{*} \Omega^1_{E/K}$ with $\pi\colon E \to \Spec(K)$ being the structure morphism.

\begin{prop}\label{global_chi_prop} 
Let $E/K$ be a non-isotrivial elliptic curve. There exists a dense open subset $U$ of $C$ such that $E$ has everywhere good ordinary reduction and $\chi(U,\mathcal{E}[F])$ is a well-defined square in $\QQ^{\times}$.
\end{prop}
\begin{proof} 
By the Oort-Tate classification, $E[F]/K$ is isomorphic to $N_{\oomega^{-1},0,\alpha}$. By Proposition B.4 in~\cite{milne} and its proof, it extends to a finite flat group scheme ${\mathcal N}/C$ of order $p$ of the form $N_{\mathcal{O}_C(W),0,\alpha}$ for some Weil divisor $W\leqslant 0$ such that $(\alpha)\geqslant W$. Let $U_1$ be a dense open subset of $C$ over which $\mathcal{E}$ has good reduction. As in the proof of Theorem~III.8.2 in~\cite{milne} on page~291, we replace $U_1$ by a smaller open set $U_2$, over which $\mathcal{N}\vert_{U_2}\simeq\mathcal{E}[F]\vert_{U_2}$. Finally, we set $U$ equal to the open subset of $U_2$ where we have removed all places $v$ for which $E/K$ has good supersingular reduction.

Write $\mathcal{N}^{D}$ for the Cartier dual of $\mathcal{N}$.
By Proposition~III.0.4.(c) and Remark~III.0.6.(b) in~\cite{milne}, we have a long exact sequence
\begin{equation*}
 \xymatrix@1{\cdots \ar[r]& \Hflc^i\bigl(U,\mathcal{N}^{D} \bigr)\ar[r] &
  \Hfl^i\bigl(C, \mathcal{N}^D\bigr)\ar[r] &
  \prod_{v\not\in U} \Hfl^i\bigl(O_v, \mathcal{N}^D\bigr)\ar[r]&
  \cdots.}
\end{equation*}
 Global duality (Theorem~III.8.2 in~\cite{milne}) shows that 
 \begin{equation*}
   \Hflc^i\bigl(U,\mathcal{N}^{D} \bigr) = \Hflc^i\bigl(U, \mathcal{E}'[V]\bigr) \quad\text{ is dual to}\quad 
   \Hfl^i\bigl(U, \mathcal{E}[F]\bigr).
 \end{equation*}
 By the multiplicative property of the Euler characteristic, we get
 \begin{equation*}
    \chi(U,\mathcal{E}[F])=\frac{\chi(C,\mathcal{N}^D)}{\prod_{v\not\in U}\chi(O_v,\mathcal{N}^D)}.
 \end{equation*}
 Since $\mathcal{N}^{D} = N_{\mathcal{O}_C(-W), \alpha, 0}$ is finite flat of order $p$ over $C$, 
 Lemma~\ref{globalchi_lem} shows that $\chi(C,\mathcal{N}^D)$ is a square.
 Furthermore, Lemma~\ref{localchi_lem} yields 
 \begin{equation*}
   \chi(U,\mathcal{E}[F]) \equiv \prod_{v\not\in U}\chi(O_v,\mathcal{N}^D)^{-1} \equiv \prod_{v\not\in U}\vert\alpha\vert_v \pmod{\square}.
 \end{equation*}
  Since the places of $U$ are places of good ordinary reduction where $\vert\alpha\vert_v$ is a square by Proposition~\ref{good_prop}, we have, using the product formula,
  \begin{equation*}
    \chi(U,\mathcal{E}[F]) \equiv \prod_{v\not\in U}\vert\alpha\vert_v^{-1}\equiv\prod_{v}\vert\alpha\vert_v^{-1}=1\pmod{\square}.
    \qedhere
  \end{equation*}
\end{proof}

%
%

\section{The Cassels-Tate pairing}\label{cassels-tate}

 Recall that there exist a pairing (proof of Theorem II.5.6 in~\cite{milne}) called the Cassels-Tate pairing 
\begin{equation*}
\langle\!\langle\cdot,\cdot\rangle\!\rangle \colon \Sha(E/K)\times \Sha(E/K)\to {}^{\QQ}\!/\!{}_{\ZZ}.
\end{equation*}
As claimed in Proposition~III.9.5 in~\cite{milne} its left and right kernels are the divisible part $\divi (\Sha(E/K))$ of the Tate-Shafarevich group. We are calling the attention of the reader to the fact that the initial proof in~\cite{milne} is wrong as noticed by D. Harari
and T. Szamuely in~\cite{HS}. The first correct
published proofs that the Cassels-Tate pairing of~\cite{milne}, Theorem II.5.6(a), annihilates only maximal divisible subgroups appear in~\cite{HS} (for prime-to-$p$ primary components) and in~\cite{Go} (for
$p$-primary components) when the 1-motive considered in these references is taken to be ($0\to  E$). This pairing is alternating and hence the order of $\Sha(E/K)_{\divi}$ is a square. This last fact is not always true if we consider general abelian varieties.

\begin{lem}\label{adjoint_lem}
 Let $\phi\colon E\to E'$ be an isogeny of elliptic curves and $\hat \phi$ the dual isogeny. Then the induced map $\phiSha\colon \Sha(E/K)\to\Sha(E'/K)$ and $\hatphiSha\colon\Sha(E'/K)\to\Sha(E/K)$ are adjoints with respect to the Cassels-Tate pairings, i.e.
$$\langle\!\langle\phiSha(\eta),\xi\rangle\!\rangle_{E'}=\langle\!\langle \eta,\hat\phiSha(\xi)\rangle\!\rangle_{E}$$
for every $\eta\in \Sha(E/K)$ and $\xi\in \Sha(E'/K)$.
\end{lem}
\begin{proof} 
 The proof is analogous to the proof in the number field case (see Remark~I.6.10 in~\cite{milne} or \S 2 of~\cite{cassels}) and is deduced from the functoriality of the local pairings in flat cohomology.
\end{proof}
\begin{prop}\label{orthongal_prop}
 The orthogonal complement of $\Sha(E'/K)[V]$ in $\Sha(E'/K)[p^{\infty}]$ under the Cassels-Tate pairing 
$$\Sha(E'/K)[p^\infty]\times \Sha(E'/K)[p^\infty]\to {}^{\QQ}\!/\!{}_{\ZZ}$$ is the image of $F\colon \Sha(E/K)[p^{\infty}]\to \Sha(E'/K)[p^{\infty}]$.
\end{prop}
\begin{proof} Note that the proposition follows immediately from the previous lemma if the pairing is perfect. Else, by the previous Lemma~\ref{adjoint_lem}, it is immediate that $F\bigl(\Sha(E/K)[p^{\infty}]\bigr)$ is contained in the orthogonal of $\Sha(E'/K)[V]$. Let $\xi$ be an element in $\Sha(E'/K)[p^{\infty}]$ orthogonal to the kernel of $V$. 
  Let $D'$ denote the maximal divisible subgroup of $\Sha(E'/K)[p^\infty]$ and $D$ the maximal divisible subgroup of $\Sha(E/K)[p^\infty]$. Then there is a perfect paring on the quotients $\Sha(E'/K)[p^{\infty}]/D'$ and $\Sha(E/K)[p^{\infty}]/D$. Since $V$ and $F$ map divisible elements to divisible elements, they induce maps between these quotients.
  \begin{equation*}
    \xymatrix{%
    0 \ar[r] & D' \ar[r] \ar@<1ex>[d]^{V}& \Sha(E'/K)[p^{\infty}] \ar[r]\ar@<1ex>[d]^{V} & \Sha(E'/K)[p^{\infty}]/D' \ar[r]\ar@<1ex>[d]^{V} & 0 \\
    0 \ar[r] & D \ar[r] \ar@<1ex>[u]^{F}& \Sha(E/K)[p^{\infty}] \ar[r]\ar@<1ex>[u]^{F} & \Sha(E/K)[p^{\infty}]/D \ar[r]\ar@<1ex>[u]^{F} & 0 }
  \end{equation*}
  The element $\xi + D'$ in the quotient $\Sha(E'/K)[p^{\infty}]/D'$ is orthogonal to the kernel of $V$. Since the pairing is perfect there, we have an element $\eta$ in $\Sha(E/K)[p^{\infty}]$ such that $F$ maps $\eta+D$ to $\xi+D'$ in the quotients.  Hence $F(\eta) = \xi + \delta$ for some $\delta \in D'$.
  But since the map $F\circ V = [p]$ is surjective on $D'$, the map $F$ maps $D$ onto $D'$. Hence $\delta$ is in the image of $F$ and so is $\xi$. 
\end{proof}

The short exact sequence of finite flat group schemes
\begin{equation}\label{ses5_seq}
 0 \to E[F]\to E[p] \to E'[V] \to 0,
\end{equation}
induces, when passing to flat cohomology, the top row of the following exact commutative diagram.
\begin{equation*}
 \xymatrix@R=4mm{
  \dots\ar[r] & E'(K)[V]\ar[r] & \Hfl^1(K,E[F]) \ar[r]\ar[d] & \Hfl^1(K,E[p])\ar[r]^F \ar[d] & \Hfl^1(K,E'[V]) \ar[d] \\
              & 0 \ar[r] & \prod_v \Hfl^1(K_v,E)[F] \ar[r]   &\prod_v \Hfl^1(K_v,E)[p] \ar[r]^F &\prod_v \Hfl^1(K_v,E')[V] 
  }
\end{equation*}
From the above diagram, we obtain an exact sequence
\begin{equation}\label{ses6_seq}
  \xymatrix@R=3mm{
  0 \ar[r] & E(K)[F] \ar[r]    & E(K)[p] \ar[r]^{F}                & E'(K)[V] \ar[r]     &                            &\\
    \ar[r] & \Sel_F(E/K)\ar[r] & \Sel_p(E/K) \ar[r]^{F} & \Sel_V(E'/K) \ar[r] &  T \ar[r] & 0,
  }
\end{equation}
where $T$ is the cokernel of the map induced by $F$ on the Selmer groups. 
Parallel to this, we have a long exact (kernel-cokerel) sequence
\begin{equation}\label{kcok_seq}
  \xymatrix@R=3mm{
  0 \ar[r] & E(K)[F] \ar[r]    & E(K)[p] \ar[r]^{F}      & E'(K)[V] \ar[r]  &\\
    \ar[r] & E'(K)/F(E(K)) \ar[r]^V & E(K)/p E(K) \ar[r] & E(K)/V(E'(K)) \ar[r] &  0,
  }
\end{equation}
We may quotient the exact sequence~\eqref{ses6_seq} by the exact sequence~\eqref{kcok_seq},
using Kummer maps in the short exact sequence~\eqref{selsha_seq}. We get an alternative description of $T$ by an exact sequence.
\begin{equation}\label{ses7_seq}
  \xymatrix@1{
  0 \ar[r] & \Sha(E/K)[F]\ar[r] & \Sha(E/K)[p] \ar[r]^{F} &  \Sha(E'/K)[V] \ar[r] & T\ar[r] & 0.
  }
\end{equation}

\begin{cor}\label{T_square_cor}
  Let $E/K$ be an elliptic curve. The order of $T$ is a square. In other words,
  \begin{equation*}
    \#\Sha(E/K)[F]\cdot \#\Sha(E'/K)[V] \equiv \Sha(E/K)[p] \pmod{\square}.
  \end{equation*}
\end{cor}
\begin{proof} 
By restriction the Cassels-Tate pairing induces a pairing on $\Sha(E'/K)[V]$ with values in $\cyclic{p}$. By the previous proposition the right and left kernels of this pairing are equal to the intersection of $F\bigl(\Sha(E/K)[p^{\infty}]\bigr)$ and $\Sha(E'/K)[V]$, which is equal to $F\bigl(\Sha(E/K)[p]\bigr)$. So the pairing induces a non-degenerate alternating pairing on $T$; hence the order of $T$ is a square. 
\end{proof}

\begin{lem}\label{rp_lem}
 We have
\begin{equation*}
p^{r_p} \equiv \frac{\#E(K)[F] \cdot \#\Sel_V(E'/K)}{\#E'(K)[V] \cdot \#\Sel_F(E/K)}  \pmod{\square}.
\end{equation*}
\end{lem}
Of course, we have $\# E(K)[F] = 1$, but we include it here so as to make the formula resemble the symmetric formula in the classical case, like in Fisher's appendix to~\cite{dok_nonab}. 
\begin{proof}
 By the short exact sequence~\eqref{ses1_seq}, $r_p=r+ \corank_{\ZZ_p}\Sha(E/K)[p^\infty]$, where $r=\rank_{\ZZ}(E(K))$ and $r_p$ is the $\ZZ_p$-rank of the dual of $\Sel_{p^\infty}(E/K)$.
 Now, since 
 $\Sha(E/K)[p^\infty]$ is cofinitely generated as $\ZZ_p$-module, we have 
\begin{equation*}
 \dim_{\FF_p} \Sha(E/K)[p] = \corank_{\ZZ_p} \bigl(\divi \Sha(E/K) [p^\infty]\bigr) + \dim_{\FF_p} \bigl( \Sha(E/K)_{\divi} [p] \bigr)
\end{equation*}
 As noticed at the beginning of section~\ref{cassels-tate}, $\#\Sha(E/K)_{\divi}$ (and therefore $\#\Sha(E/K)_{\divi}[p]$) is a square. We deduce that
 \begin{equation*}
  r_p \equiv r+\dim_{\FF_p} \Sha(E/K)[p]  \pmod{2}.
 \end{equation*}
 On the other hand, the short exact sequence~\eqref{selsha_seq} applied to $[p]$ implies that
 \begin{equation*}
  \dim_{\FF_p} \Sel_p(E/K) = r + \dim_{\FF_p} E(K)[p] + \dim_{\FF_p}\Sha(E/K)[p],
 \end{equation*}
since $E(K)/pE(K) \simeq E(K)[p]\oplus (\cyclic{p})^r$. So we get the formula
 \begin{equation*}
 r_p\equiv \dim_{\FF_p}E(K)[p] +\dim_{\FF_p}\Sel_p(E/K) \pmod{2}.
 \end{equation*}
 The assertion results then from the exact sequence~\eqref{ses6_seq} and Corollary~\ref{T_square_cor}.
\end{proof}

%
%

\section{Global duality}

\begin{prop}\label{global_duality_prop} 
Let $E/K$ be a non-isotrivial elliptic curve and let $U$ be an open subset of $C$ over which $E$ has good reduction.
Then we have
\begin{equation*}
\frac{\#E(K)[F] \cdot \#\Sel_V(E'/K)}{\#E'(K)[V] \cdot \#\Sel_F(E/K)} = \frac{1}{\chi(U,\mathcal{E}[F])}\cdot \prod_{v\not \in U} z(V_{E'(K_v)}).
\end{equation*}
\end{prop}
We insist once more that the roles of $F$ and $V$ here are not interchangeable, e.g. the terms $z(F_{E(K_v)})$ in the product would not be finite.
\begin{proof}
 The long exact sequence for flat cohomology deduced from the definition of $\Hflc^i$ in Proposition~III.0.4.(a) in~\cite{milne} reads
 \begin{equation*}
 \xymatrix@1{\cdots \ar[r]& \Hflc^i(U,\cdot)\ar[r]&\Hfl^i(U,\cdot)\ar[r]& \bigoplus_{v\not\in U} \Hfl^i(K_v,\cdot)\ar[r]& \Hflc^{i+1}(U,\cdot)\ar[r]&\cdots}
 \end{equation*}
 \color{black}The global duality in Theorem~III.8.2 of~\cite{milne} implies that the group $\Hflc^i(U, \mathcal{E}[F])$ is dual to $\Hfl^{3-i}(U,\mathcal{E}'[V])$ since $\mathcal{E}[F]$ is finite and flat over $U$. We find the following long exact sequence
 \begin{equation*}
   \xymatrix@R-1ex{%
          & \Hfl^1\bigl(U,\mathcal{E}[F]\bigr) \ar[r]& \bigoplus_{v \not \in U} \Hfl^1\bigl(K_v, E[F]\bigr) \ar[r]& \Hfl^1\bigl(U,\mathcal{E}'[V]\bigr)^{\vee} \ar[r]& \\
  \ar[r]  & \Hfl^2\bigl(U,\mathcal{E}[F]\bigr) \ar[r]& \bigoplus_{v \not \in U} \Hfl^2\bigl(K_v, E[F]\bigr) \ar[r]& \Hfl^0\bigl(U,\mathcal{E}'[V]\bigr)^{\vee} \ar[r]& 0
   }
 \end{equation*}
 Local duality as in Theorem~III.6.10 in~\cite{milne} shows that $\Hfl^2(K_v,E[F])$ is dual to $E'(K_v)[V]$. 
 Our aim is to replace the local term $\Hfl^1(K_v,E[F])$ by the cokernel of the map from $E'(K_v)/F(E(K_v))$. By local duality (Theorem~III.7.8 in~\cite{milne} and the functoriality of biextensions), this term is dual to $\Hfl^1(K_v,E')[V]$. So we will quotient the term $\Hfl^1\bigl(U,\mathcal{E}'[V]\bigr)^{\vee}$ by the image of the map on the right hand side in the following commutative diagram.
\begin{equation}\label{useless_eq}
\xymatrix@R=9mm{
  & \bigoplus_{v \not\in U} {}^{E'(K_v)}\!/\!{}_{F(E(K_v))} \ar[r]^{\cong} \ar[d] & \bigoplus_{v \not \in U} \bigl(\Hfl^1(K_v,E')[V]\bigr)^{\vee} \ar[d] & \\
  \Hfl^1(U,\mathcal{E}[F]) \ar[r] & \bigoplus_{v \not \in U} \Hfl^1\bigl(K_v, E[F]\bigr) \ar[r]& \Hfl^1\bigl(U,\mathcal{E}'[V]\bigr)^{\vee} \ar[r] & \dots
 }
\end{equation}
Because of the exact Kummer sequence  
\begin{equation*}
   \xymatrix@1{ 0\ar[r] & {}^{E'(K_v)}\!/\!{}_{F(E(K_v))} \ar[r] & \Hfl^1(K_v,E[F]) \ar[r] & \Hfl^1(K_v,E)[F] \ar[r] & 0. }
\end{equation*}
the cokernel of the map on the left in~\eqref{useless_eq} is $\bigoplus_{v\not\in U}\Hfl^1(K_v,E)[F]$, which, again by local duality, is dual to  $\bigoplus_{v \not \in U} E(K_v)/V(E'(K_v))$. By definition the cokernel of the map on the right in~\eqref{useless_eq} is the dual of the Selmer group $\Sel_V(E'/K)$.

%
 Putting all these results together, we obtain the long exact sequence
 \begin{equation*}
   \xymatrix@R-2ex{%
   0\ar[r] & \Sel_F(E/K) \ar[r] & \Hfl^1\bigl(U,\mathcal{E}[F]\bigr) \ar[r]& \bigoplus_{v \not \in U} \Bigl({}^{E(K_v)}\!/\!{}_{V(E'(K_v))}\Bigr)^{\vee} \ar[r] & \\
    \ar[r] & \Sel_V(E'/K)^{\vee} \ar[r] &  \Hfl^2\bigl(U,\mathcal{E}[F]\bigr) \ar[r]& \bigoplus_{v \not \in U} \Bigl(E'(K_v)[V]\Bigr)^{\vee} \ar[r] & \\
    \ar[r] & \Bigl(E(K)[V]\Bigr)^{\vee} \ar[r]& 0
   }
 \end{equation*}
 Since all other terms in the sequence are finite, the groups $\Hfl^i(U,\mathcal{E}[F])$ are finite, too. The alternating product of its orders gives the result.
\end{proof} 

If $E/K_v$ is a non-isotrivial, semistable elliptic curve then one can show that the group scheme $\mathcal{E}[F]$ is finite and flat. So  the result of Proposition~\ref{global_duality_prop} can be extended to any open subset $U$ such that $E$ has semistable reduction over all places in $U$. In particular $U$ can be taken to be equal to $C$, if $E/K$ is semistable.

%
%

\section{The proof of the $p$-parity}\label{pparity_sec}

We now pass to the proof of Theorem~\ref{pparity_thm}. We return now to our running assumptions. $K$ has characteristic $p>3$ and $E/K$ is not isotrivial. We present first the main results coming from global duality and the local computations and then we just have to put them together. But both these statements are interesting in their own right.

\begin{thm}\label{parity_corank_thm}
 Let $E/K$ be a non-isotrivial elliptic curve. We have
 $$
  p^{r_p} \equiv \prod_v z\Bigl( V_{E'(K_v)}\colon E'(K_v)\to E(K_v)\Bigr) \pmod{\square}
 $$
 where the product runs over all places $v$ in $K$.
\end{thm}
\begin{proof}
  Proposition~\ref{global_chi_prop} provides us with an open subset $U$ in $C$ such that $E$ has good ordinary reduction at all places in $U$.
  It follows from Lemma~\ref{rp_lem}, Proposition~\ref{global_duality_prop}, and Proposition~\ref{global_chi_prop} that
  \begin{align*}
    p^{r_p} &\equiv  \frac{\#E(K)[F] \cdot \#\Sel_V(E'/K)}{\#E'(K)[V] \cdot \#\Sel_F(E/K)}  \pmod{\square}\\
            &= \frac{1}{\chi(U,\mathcal{E}[F])}\cdot \prod_{v\not \in U} z\Bigl( V_{E'(K_v)}\colon E'(K_v)\to E(K_v)\Bigr) \\
            & \equiv \prod_{v\not \in U} z\Bigl( V_{E'(K_v)}\colon E'(K_v)\to E(K_v)\Bigr) \pmod{\square}
  \end{align*}
  Finally from Proposition~\ref{good_prop}, we know that $z(V_{E'(K_v)})$ is a square for all places $v \in U$ as $E$ has good ordinary reduction there.  
\end{proof}
Next, we collect from section~\ref{local_sec} the following result.
\begin{prop}\label{rootno_prop}
 Let $E/K$ be a semistable elliptic curve. Then the root number is $w(E/K) = (-1)^s$ where $s$ is the number of split
 multiplicative primes for $E/K$. Furthermore $s$ has the same parity as the $p$-adic valuation of 
 \begin{equation*}
  \prod_{v} \frac{ c_v(E/K) }{c_v(E'/K)}
 \end{equation*}
 where $c_v$ are Tamagawa numbers.
\end{prop}
\begin{proof}
 From Theorem~\ref{local_thm} we deduce that 
 \begin{equation*}
  w(E/K)  = \prod_v w(E/K_v) = \prod_v \sigma(E/K_v) \cdot \Bigl(\frac{-1}{L_w/K_v}\Bigr) = \prod_v \sigma(E/K_v)
 \end{equation*}
 by the product formula for the norm symbols $\prod_v \bigl(\frac{-1}{L_w/K_v}\bigr)$ with $L$ being the extension of $K$ over which $E[F]=\mu_p$ and $w$ is any place above $v$. Using the Propositions~\ref{good_prop}, \ref{split_prop}, and~\ref{nonsplit_prop}, we see that $\sigma(E/K_v)$ is  $-1$ if and only if $E$ has split multiplicative reduction at $v$.

 If the reduction at $v$ is split multiplicative, then we have $c_v(E'/K) = p\cdot c_v(E/K)$ since the parameters in the Tate parametrisation satisfy $q_{E'} = {q_E}^p$. If the reduction is non-split multiplicative
 then the Tamagawa numbers can only be 1 or 2.
\end{proof}
Note that we could have used the known modularity and the Atkin-Lehner operators to prove this statement without the computations in section~\ref{local_sec}, at least if $E$ has at least one place of split multiplicative reduction.
\begin{proof}[Proof of Theorem~\ref{pparity_thm}]
 First, we use Corollary~\ref{toss_cor} which allows us to assume that $E/K$ is semistable.
 Then by the previous Proposition~\ref{rootno_prop} we have $w(E/K) = \prod_v \sigma(E/K_v)$ and Theorem~\ref{parity_corank_thm} states
 that $(-1)^{r_p}= \prod_v \sigma(E/K_v)$.
\end{proof}
%

%
%

\section{Local Root Number Formula}\label{rootno_sec}

We now prove Theorem~\ref{local_thm} without any hypothesis on the reduction. We use, without repeating the definitions, the notations from section~\ref{local_sec}.
\begin{thm}\label{local2_thm}
 Let $K$ be a local field of characteristic $p>3$.
 For any non-isotrivial elliptic curve $E/K$, we have $w(E/K) = \bigl(\frac{-1}{L/K}\bigr) \cdot \sigma(E/K)$.
\end{thm}
As mentioned in section~\ref{local_sec} this answers positively a conjecture in~\cite{dok_09} for the isogeny $V$. This theorem could certainly be shown by local computations only, but they would tend to be very tedious for additive potentially supersingular reduction. We can avoid this here by using a global argument. This is a similar idea as in the proof of Theorem~5.7 in~\cite{dok_09}.
\begin{proof}
 By Theorem~\ref{local_thm}, we may assume that $E/K$ has additive, potentially good reduction. Let $n\geqslant 12$.
  We can find a minimal integral equation $y^2 = x^3 +A\,x + B$ for $E/K$. Choose a global field $\mathcal{K}$ of characteristic $p$ with a place $v_0$ such that $\mathcal{K}_{v_0} = K$. 
 Choose another place $v_1 \neq v_0$ in $\mathcal{K}$ and choose a large even integer $N$ such that $(N-1)\cdot \deg(v_1) > 2g - 1 + n\deg(v_0)$, where $g$ is the genus of $\mathcal{K}$. For a divisor $D$ on the projective smooth curve $\mathcal{C}$ corresponding to $\mathcal{K}$, we write $L(D)$ for the Riemann-Roch space
 $H^0(\mathcal{C},\mathcal{O}_C(D))$.
The inequality on $N$ guarantees that the dimensions of the Riemann-Roch spaces in the exact sequence
 \begin{equation*}
   \xymatrix@1{ 0 \ar[r]& L\bigl(N (v_1) - n (v_0) \bigr) \ar[r]&  L\bigl(N (v_1) \bigr) \ar[r] & \mathcal{O}_{v_0}/\mathfrak{m}_{v_0}^n \ar[r]& 0 }
 \end{equation*}
are positive -- e.g. equal to $N\deg(v_1) -n\deg(v_0) + 1 -g> g+\deg(v_1)$ for the smaller space. Choose an element $a$ in $L(N (v_1))$ which maps to $A+\mathfrak{m}_{v_0}^n$ on the right. We can even impose that it does not lie in $L((N-1)(v_1))$, since this is a subspace of codimension $\deg(v_1)>0$ in $L(N(v_1))$. Then $a$ has a single pole of order $N$ at $v_1$ and it satisfies $v_0(A-a) \geq n$.
Next, we use that $N$ is even and we choose an element $b$ in $\mathcal{K}$ such that $v_1(b) = -\tfrac{3}{2}N$, and $v_0(B-b)\geq n$. We can also impose that 
$v_1(4 a^3 + 27 b^2) > -3N$. Furthermore we impose that the zeroes of $b$ are distinct from the zeroes of $a$; this excludes at worst $N\deg(v_1)$ subspaces of codimension $1$ in $L\bigl(\tfrac{3N}{2}(v_1)\bigr)$.

 Let $\mathcal{E}/\mathcal{K}$ be the elliptic curve given by $y^2=x^3 +a\,x +b$.
 By the congruences on $a$ and $b$ at $v_0$ and the continuity of Tate's algorithm, the reduction of $\mathcal{E}$ at $v_0$ is additive, potentially good. 
 At the place $v_1$ the valuation of the $j$-invariant $j(\mathcal{E}) = 2^8\cdot 3^3\cdot a^3 /(4 \,a^3+27\, b^2)$ will be negative by our choices. Hence the reduction is either multiplicative or potentially multiplicative. For any other place $v$ with $v(a)>0$, we have $v(4 \,a^3+27\, b^2) = 0$ and hence the curve has good reduction at $v$, and for any other place $v$ with $v(a) = 0$, either the reduction is good or $v(j(\mathcal{E}))< 0$.

 Therefore we have constructed an elliptic curve $\mathcal{E}/\mathcal{K}$ with a single place $v_0$ of additive, potentially good reduction.
 So for all other places Theorem~\ref{local_thm} applies. 
 Let $\mathcal{L}$ be the extension of $\mathcal{K}$ such that $E[F]\cong\mu_p$ over $\mathcal{L}$. 
 Now we use the results of Theorem~\ref{parity_corank_thm} and the proven $p$-parity in Theorem~\ref{pparity_thm} to compute
 \begin{align*}
  w(\mathcal{E}/K) &= \frac{w(\mathcal{E}/\mathcal{K})}{\prod_{v\neq v_0} w(\mathcal{E}/\mathcal{K}_{v})}
                   = \frac{(-1)^{r_p}}{
                       \prod_{v\neq v_0}\bigl(\frac{-1}{\mathcal{L}_w/\mathcal{K}_v}\bigr)
                                        \sigma(\mathcal{E}/\mathcal{K}_{v})} \\
                 & =  \frac{\bigl(\frac{-1}{\mathcal{L}_{w_0}/K}\bigr)}{
                            \prod_{\textnormal{all }v} \bigl(\frac{-1}{\mathcal{L}_w/\mathcal{K}_v}\bigr)}
                   \cdot\frac{\prod_{\textnormal{all }v} \sigma(\mathcal{E}/\mathcal{K}_{v})}{
                              \prod_{v\neq v_0}\sigma(\mathcal{E}/\mathcal{K}_{v})}
   =  \Bigl(\frac{-1}{\mathcal{L}_{w_0}/K}\Bigr)\cdot\sigma(\mathcal{E}/K).
 \end{align*}
 Once again we used the product formula for the norm symbol. Now we argue that the three terms are all continuous in the topology of $K$ as $a$ and $b$ varies: For the local root number this is exactly the statement of Proposition~4.2 in~\cite{helfgott}. The field $\mathcal{L}_{w_0}$ and the order of the kernel $E'(K)[V]$ of Verschiebung are locally constant because they are defined by continuously varying separable polynomials. Finally, the order of the cokernel of $V\colon E'(K)\to E(K)$ is locally constant, because the group of connected components and the reduction and the induced map $V$ on them will not change and on the formal group the cokernel is determined by the valuation of the Hasse invariant (which again is a polynomial in $a$ and $b$) by the argument in the proof of Proposition~\ref{good_prop}.
  Since all three terms take value $\pm 1$, they will eventually, for big enough $n$, be equal to the corresponding values for $E$.
\end{proof}

%
%

\section{On the $\ell$-parity conjecture}\label{ell_sec}

We switch now to investigating the $\ell$-parity conjecture when $\ell\neq p$. As mentioned in the introduction, we have only a partial result in this case.
Recall that $p>3$ is a prime and that $K$ is a global field of characteristic $p$ with constant field $\FF_q$. 

For any $n$ and any extension $L$ of $K$, we denote by $L_n$ the field $L\cdot \FF_{q^n}$.
The aim of this section is to show the following partial result (given as Theorem~\ref{ellparity_thm} in the introduction).
\begin{thm}\label{ell_thm}
 Let $E/K$ be an elliptic curve and let $\ell$ be an odd prime different from $p$.
 Furthermore assume that 
 \begin{enumerate}
  \item\label{a_as} $a=[K(\mu_{\ell}):K]$ is even and that
  \item\label{b_as} the analytic rank of $E$ does not grow by more than $1$ in the extension $K_2/K$.
 \end{enumerate}
 Then the $\ell$-parity conjecture holds for $E/K$.
\end{thm}

Note first that we believe that~(\ref{a_as}) holds for roughly two thirds of the $\ell$ as $a$ is also the order of $q$ in the group $(\cyclic{\ell})^{\times}$. The second condition should hold quite often as it says that the analytic rank of the twist of $E$ by the unramified quadratic character is less or equal to $1$. So for instance if $b$ as in Lemma~11.3.1 in~\cite{ulmer_gnv} is odd, then~(\ref{b_as}) holds.

See the next section for a discussion about why we were not able to extend the proof here to any situation without these hypotheses. 

\begin{proof}
Corollary~\ref{toss_cor}  allows us to assume that $E$ is semistable and we may assume that $E$ is not isotrivial as for isotrivial curves even BSD is known.
First we use a non-vanishing result, to produce from the analytic information a useful extension of $K$, which we want to link to the algebraic side later.

We write $\mathfrak{n}$ for the conductor of $E/K$.
The degree of $\mathfrak{n}$ is linked to the degree of the polynomial $L(E/K,T)$ in $T=q^{-s}$ by the formula of Grothendieck-Ogg-Shafarevich (as used in formula~(5.1) of~\cite{ulmer_analogies}):
\begin{equation*}
 \deg (\mathfrak{n}) =  \deg \bigl(L(E/K,T)\bigr) - 2(2g_{K} -2).
\end{equation*}
We can factor the polynomial to
\begin{equation*}
 L(E/K,T) = (1-qT)^r \cdot (1+qT)^{r'}\cdot \prod_{i} (1-\alpha_i T)(1-\bar\alpha_i T)
\end{equation*}
where $\alpha_i$ are non-real, complex numbers of absolute value $q$. By definition $r$ is the analytic rank of $E/K$ and it is easy to see that $r+r'$ is the analytic rank of $E/K_2$ since the analytic rank of $E/K_2$ is the number of inverse zeroes $\alpha$ of $L(E/K,T)$ such that $\alpha^2 = q^2$. So we get
\begin{equation}\label{nran_eq}
 \sum_{v \text{ bad}} \deg(v) = \deg(\mathfrak{n})\equiv \deg \bigl(L(E/K,T)\bigr) \equiv r+r' = \ord_{s=1} L(E/K_2,s) \pmod{2}.
\end{equation}

We are now going to use Theorem~5.2 in~\cite{ulmer_gnv} to construct suitable extensions of $K$. The argument is very similar to the proof of Step~2 in~11.4.2 of~\cite{ulmer_gnv}. The following is a very special case of this very general and powerful theorem.

\begin{thm}[Ulmer]\label{ulmer_main_thm}
 Let $K$ be a global field of characteristic $p>3$, let $S$ be a finite non-empty set of places in $K$, let $\ell\neq p$ be an odd prime, and let $E/K$ be a semistable elliptic curve. Assume that $a=[K(\mu_{\ell}):K]$ is even and suppose that the sum of the degree of the bad places not belonging to $S$ is even. Then there exists an integer $n$ coprime to $a$ and a element $z\in K_n^{\times}$ such that the extension $K_n(\sqrt[\ell]{z})/K_n$ is totally ramified at all places above $S$ and unramified at all bad places not in $S$ and such that the analytic rank of $E$ does not grow in it.
\end{thm}
\begin{proof}
 All notations and results in this proof refer to~\cite{ulmer_gnv}.
 We use Theorem~5.2.(1) with $F=K$, $\alpha_n=q^n$, $d=\ell$, $S_r = S$ and $\rho$ the symplectically self-dual representation of weight $w=1$ attached to $E$ on the Tate module $V_{\ell}(E)$ as in section~11. We can choose the sets $S_s$ and $S_i$ arbitrarily as long as we make sure that $S$, $S_s$, and $S_i$ are disjoint. 
 The conditions (especially from his section~3.1) are satisfied. 
 Let $o$ be an orbit in $(\cyclic{\ell})^{\times}$ for the multiplication by $q$. Then $d_{o}=\ell$ and $a_{o} = a$. So we can conclude the existence of $n$ and $z$ such that $L(\rho\otimes\sigma_{o,z},K_n,T)$ does not have $\alpha_n$ as an inverse root in~5.2.(1) unless we are in the exceptional cases (i) to (iv) in~5.1.1.1. Now, (iv) cannot hold because $\rho$ is not orthogonally self-dual and (i) and (ii) are impossible because $d=\ell$ is odd. However, all the condition in (iii) are satisfied apart from maybe the condition~4.2.3.1. (In particular, we know that $-o=o$ because $a$ is even.)

 We now have to show that the hypothesis on $S$ imposes that the condition~4.2.3.1 fails. Since $E$ is semistable, the local exponent of the conductor $\cond_v(\rho)$ is $1$. Let $v$ be a bad place in $S$ and $\chi_v$ be a totally ramified character of the decomposition group $D_v$ which has exact order $\ell$. Then the conductor $\cond_v(\rho\otimes\chi_v)=2$ again because $E$ has multiplicative reduction at $v$. So the first condition in~4.2.3.1 saying that this has constant parity as $\chi_v$ varies is always fulfilled. In order to make the condition~4.2.3.1 fail, we must have that
 \begin{equation*}
  \sum_{\text{bad }v \in S} \cond_v(\rho\otimes\chi_v)\deg(v) + \sum_{\text{bad } v \not\in S}  \cond_v(\rho) \deg(v)
 \end{equation*}
is even. That is exactly what the hypothesis in the theorem imposes.
\end{proof}

\begin{lem}\label{ellext_lem}
 To prove Theorem~\ref{ell_thm}, 
 we may assume that there exists a non-constant Kummer extension $L/K$ of degree $\ell$ 
 in which the analytic rank does not grow and such that 
 \begin{itemize} 
  \item if the analytic rank of $E/K_2$ is even then no place of bad reduction ramifies in $L/K$, or
  \item if the analytic rank of $E/K_2$ is odd then exactly one place of bad reduction ramifies. Moreover, in the latter case, the degree of this place is odd.
 \end{itemize}
\end{lem}
\begin{proof}
 If the analytic rank is even we choose the finite non-empty set of places $S$ to be disjoint from the set of bad places.
 If the analytic rank is odd, then the congruence~\eqref{nran_eq} shows that there is at least one bad place $v$ of odd degree. So we choose $S$ to contain this as the only bad place.
 Then~\eqref{nran_eq} shows that the hypothesis in Theorem~\ref{ulmer_main_thm} with the above choice for $S$ holds. So we have an integer $n$ and an element $z \in K_n^{\times}$. Now we use the first item in Proposition~\ref{red_prop} to replace $K$ by its odd Galois extension $K_n$. So $L=K(\sqrt[\ell]{z})$ is the requested extension.
 \end{proof}

We now come to the algebraic part of the argument. Using the previous two lemmata, we have now a Kummer extension $L/K$ of degree $\ell$ in which the analytic rank does not grow. The Galois closure of $L/K$ is $L_a$ containing $L_2$. We have the following picture of extensions
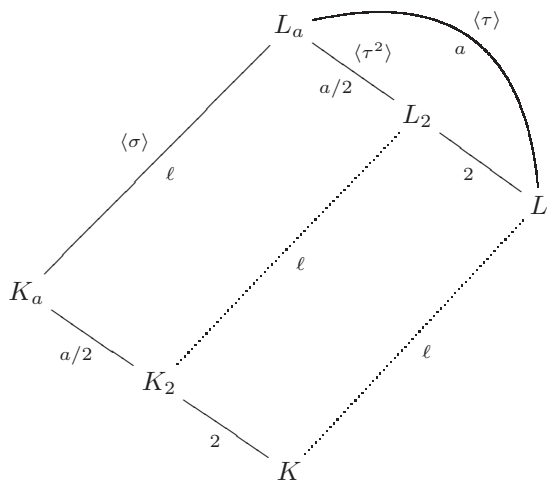
\begin{equation*}
 \xymatrix@R-1ex@C+1.5ex{
 && L_a \ar@{-}[dddll]_{\langle\sigma\rangle}^{\ell} 
        \ar@{-}[rd]^{\langle\tau^2\rangle}_{a/2}
        \ar@{-}@/^3pc/[rrdd]_{a}^{\langle\tau\rangle}&& \\
 &&& L_2 \ar@{.}[dddll]^{\ell} \ar@{-}[rd]_{2} & \\
 &&&& L \ar@{.}[dddll]^{\ell} \\
 K_a \ar@{-}[rd]_{a/2} &&&&\\
 & K_2 \ar@{-}[rd]_{2}&&& \\
 && K &&
 }
\end{equation*}
The dotted lines are non-Galois extensions. We have written the degree under each inclusion. The Galois group $G=\Gal(L_a/K)$ is a meta-cyclic group generated by elements $\sigma$ and $\tau$ of order $\ell$ and $a$ respectively, with $L=(L_a)^{\tau}$. We have 
\begin{equation*}
 G = \bigl\langle \sigma,\tau\bigl\vert \tau^a=\sigma^{\ell} = 1, \tau\sigma\tau^{-1} = \sigma^q\bigr\rangle.
\end{equation*}

We list the irreducible $\QQ_{\ell}[G]$-modules. By $\mathbbm{1}$ we denote the trivial representation. Fix a primitive character $\chi\colon\langle\tau\rangle \cong \cyclic{a}\cdot\tau\to \QQ_{\ell}^{\times}$ that we can view as a character of $G$ by setting $\chi(\sigma)=1$. (Note that $a$ divides $\ell-1$, so $\chi$ is indeed realisable over $\QQ_{\ell}$.) The non-trivial $1$-dimensional representations of $G$ are exactly the $\chi^i$ for $1\leqslant i \leqslant a-1$. There is only one non-trivial irreducible $\QQ_{\ell}[\langle\sigma\rangle]$-module. It is of degree $\ell-1$.  We can represent it as $\rho = \QQ_{\ell}[\xi]$ where $\xi$ is a primitive $\ell$\textsuperscript{th} root of unity and $\sigma$ acts on $\rho$ by multiplication with $\xi$. (Over $\bar\QQ_{\ell}$ is would split into the $\ell-1$ non-trivial characters of $\langle\sigma\rangle\cong\cyclic{\ell}$.) 
We make $\rho$ into a $G$-module by defining the $\QQ_{\ell}$-linear action of $\tau$ by $\tau(\xi^j) = \xi^{qj}$ for all $0\leqslant j\leqslant \ell-2$. It is easy to see that $\rho$ is an irreducible $\QQ_\ell[G]$-module of degree $\ell-1$ and in fact it is the only higher dimensional irreducible $\QQ_\ell[G]$-module. (Note that $\rho\otimes\overline\QQ_{\ell}$ decomposes into $\tfrac{\ell-1}{a}$ irreducibles of degree $a$ corresponding to the orbits of the multiplication by $q$ on $(\cyclic{\ell})^{\times}$.) We have 
\begin{equation*}
 \QQ_{\ell}[G] = \mathbbm{1} \oplus\bigoplus_{i=1}^{a-1} \chi^i \oplus \rho^a.
\end{equation*}
For convenience we will denote $\chi^{a/2}$ by $\varepsilon$. The fixed field of the kernel of $\varepsilon$ is $K_2$.

To announce the next lemma, we need to introduce the corrected product of Tamagawa numbers.
Fix the invariant $1$-form $\omega$ on $E/K$ corresponding to the fixed Weierstrass equation. For each place $v$, write $c_v(E/K)$ for the Tamagawa number and define 
\begin{equation*}
 C_v(E/K,\omega) = c_v(E/K)\cdot \Biggl\vert \frac{\omega}{\omega^{o}_{v}}\Biggr\vert_v
\end{equation*}
where $\omega^o_v$ is a N\'eron differential for $E/K_v$. The global product over all places $v$ of $K$
\begin{equation*}
 C(E/K) = \prod_v C_v(E/K,\omega)
\end{equation*}
is no longer dependent on the choice of $\omega$ by the product formula.

For any irreducible $\QQ_\ell[G]$-module $\psi$, write $m_{\psi}$ for the multiplicity of the $\psi$-part of the $\ell$-primary Selmer group $\Sel_{\ell^{\infty}}(E/L_2)$.
\begin{lem}
 We have
 \begin{equation*}
 m_{\mathbbm{1}} + m_{\varepsilon} + m_{\rho} \equiv \ord_{\ell} \Biggl( \frac{C(E/L_2)}{C(E/K_2)} \Biggr) \pmod{2}.
 \end{equation*}
\end{lem}
\begin{proof}
We are interested in the following relation between permutations representations (in the terminology of Dokchitsers' work, say~2.3 in~\cite{dok_modsquares})
\begin{equation*}
 \Theta = 2 \cdot G + \langle \tau^2 \rangle - 2 \cdot \langle \tau\rangle - \langle \sigma,\tau^2\rangle
\end{equation*}
corresponding to the equality of $L$-functions
\begin{equation*}
 L(E/K,s)^2 \cdot L(E/L_2,s) = L(E,\mathbbm{1},s)^3 \cdot  L(E,\varepsilon,s)\cdot L(E,\rho,s)^2 
  = L(E/K_2,s) \cdot L(E/L,s)^2.
\end{equation*}
It can be seen that the regulator constants (as defined in~2.11 of~\cite{dok_modsquares}) satisfy
\begin{equation*}
 C_{\Theta}(\mathbbm{1})\equiv C_{\Theta}(\varepsilon)\equiv C_{\Theta}(\rho)\equiv\ell\pmod{\square}
\end{equation*}
in $\QQ^{\times}$ modulo squares. For $\mathbbm{1}$ and $\varepsilon$ this is straightforward; for $\rho$ we best use
Theorem 4.(4) of~\cite{dok_reg} with $D=\langle\tau\rangle$, implying that 
\begin{equation*}
 C_{\Theta}(\rho)\cdot C_{\Theta}(\mathbbm{1}) = C_{\Theta}\Bigl(\QQ_\ell[G/\langle\tau\rangle]\Bigr) = 1.
\end{equation*}
So $S_{\Theta} = \{\mathbbm{1},\varepsilon,\rho\}$ in Dokchitsers' notation in~\cite{dok_sd}.
 
In short everything looks just like if $L_2/K$ were a dihedral extension (which it is not unless $a=2$). For $a=2$ this is computed in Example~1 in~\cite{dok_reg} and Example~4.5 in~\cite{dok_modsquares} and Example~3.5 in~\cite{dok_sd}. For $a=\ell-1$, this is Example~2.20 in~\cite{dok_modsquares} and Example~3.6 in~\cite{dok_sd}.

 Now, Theorem~1.6 in~\cite{dok_sd} shows that
 \begin{equation*}
 m_{\mathbbm{1}} + m_{\varepsilon} + m_{\rho} \equiv \ord_{\ell} \Biggl( \frac{C(E/K)^2\cdot C(E/L_2)}{C(E/L)^2\cdot C(E/K_2)} \Biggr) \pmod{2}
 \end{equation*}
 which proves the lemma.
\end{proof}

\begin{lem}\label{fracc_lem}
 Suppose that no bad place ramifies in $L/K$, then
 the $\ell$-adic valuation of the integer $C(E/L_2)/C(E/K_2)$ is even.
 If there is only one bad place that ramifies in $L/K$ and this place is of odd degree, then the $\ell$-adic valuation of $C(E/L_2)/C(E/K_2)$ is odd.
\end{lem}
 The more general statement for $a=2$ can be found in Remark 4.18 in~\cite{dok_modsquares}.
\begin{proof}
 Let $v$ be a place of $K_2$. Write $y$ for $\omega/\omega^o_v$. Then
\begin{equation*}
 \frac{\prod_{w\mid v} C_w(E/L_2,\omega)}{C_v(E/K_2,\omega)}
 = \frac{\prod_{w\mid v} c_w(E/L_2)}{c_v(E/K_2)}\cdot \frac{\prod_{w\mid v} \vert y \vert_w}{\vert y \vert_v} 
 \equiv \frac{\prod_{w\mid v} c_w(E/L_2)}{c_v(E/K_2)} \pmod{\square}
\end{equation*}
because $\prod_{w\mid v} \vert y \vert_w / \vert y \vert_v = \vert y \vert_v^{\ell-1}$ is a square.
If the place $v$ is unramified, then the type of reduction and the Tamagawa number do not change and we have
\begin{equation*}
 \frac{\prod_{w\mid v} c_w(E/L_2)}{c_v(E/K_2)} =
  \begin{cases}
   c_v(E/K_2)^{\ell-1} \quad\ & \text{ if $v$ decomposes in $L_2/K_2$ and}\\
   1                          & \text{ if $v$ is inert.}
  \end{cases}
\end{equation*}
In either case it is a square. If the reduction is good at $v$ then $c_w(E/L_2) = c_v(E/K_2) = 1$. This proves the first case.

Suppose now $v$ is a place in $K_2$ which lies above a place of odd degree in $K$ and which ramifies in $L_2/K_2$. Then the place is inert in $K_2/K$ and hence the reduction of $E/K_2$ at $v$ is necessarily split multiplicative. Let $q$ be the Tate parameter of $E$ at $v$. Then $c_v(E/K_2) = v(q)$ and $c_w(E/L_2) = w(q) = \ell\cdot v(q)$ for the place $w$ above $v$. So the quotient is $\ell$ which has odd $\ell$-adic valuation. This proves the second statement.
\end{proof}

 Finally we can finish the proof of Theorem~\ref{ell_thm}. By construction, we have $\ord_{s=1} L(E,\rho,s) = 0$. As in the proof of Proposition~\ref{red_prop}, this implies that $m_{\rho} = 0$; in fact $L(E,\rho,s)$ is $L(B/K,s)$ for the extension $L/K$. So the last two lemmata show that $m_{\mathbbm{1}}+m_{\varepsilon}$, which is the corank of the $\ell$-primary Selmer group $\Sel_{\ell^{\infty}}(E/K_2)$, has the same parity as the analytic rank of $E/K_2$. This proves the $\ell$-parity conjecture for $E/K_2$. Assumption~(\ref{b_as}) and Proposition~\ref{red_prop} prove that the $\ell$-parity holds over $K$, too. 
\end{proof}

%
%

\section{Failure to extend}\label{fail_sec}
 Although it is not usual to write in a mathematical article about unsuccessful attempts to prove a result, we wish to include in this last section a short explanation of why we were unable to extend the proof in the previous section.  We hope this might be the starting point for a complete proof of the $\ell$-parity conjecture. We try to outline here the missing non-vanishing result for $L$-functions, which might be accessible using automorphic methods.

 The main ingredient for proving Theorem~\ref{ell_thm} was the existence of a Kummer extension of degree $\ell$ in which the analytic rank does not grow. Moreover this extension was linked in a ``non-commutative way'' to an even abelian extension. The machinery using representation theory set up by Tim and Vladimir Dokchitser is then sufficient to prove the parity.

 First, if condition~(\ref{b_as}) in Theorem~\ref{ell_thm} does not hold but condition~(\ref{a_as}) still holds, then there is no hope that a Kummer extension will do. 
In order to obtain a Galois extension of $K$ from a Kummer extension, we need to make the extension $K_2/K$. But without any control about the growth of the analytic rank in this quadratic extension, we do not know how to prove the $\ell$-parity over $K$. With some extra work, one can conclude that the $\ell$-parity conjecture holds for $E/K_2$. In this case, we would need a non-vanishing result for an extension of $K$ of degree dividing $\ell$ which is not a Kummer extension. 

 Suppose now that the condition~(\ref{a_as}) does not hold. Then we would need to find the ``dihedral'' extension somewhere else. The Proposition~\ref{ell_non_prop} below formulates this in a positive way.

\begin{prop}\label{ell_non_prop}
  Suppose $E/K$ is semi-stable and non-isotrivial. 
  Let $F/K$ be a quadratic extension such that the analytic rank of $E$ grows at most by $1$ in $F/K$ and the analytic rank of $E/F$ is even. 
  Assume 
 \begin{enumerate}
  \item $a=[F(\mu_{\ell}):F]$ is odd and
  \item there exists an odd $n\geqslant 1$ and a $z\in F_n^{\times}$ with the property that $L=F_n(\sqrt[\ell]{z})$ is an extension of degree $\ell$ of $F_n$ such that $L_a/K_{an}$ is a dihedral extension, no bad place of $E/F_n$ ramifies in $L/F_n$, and the analytic rank of $E$ does not grow in $L/F_n$.
 \end{enumerate}
 Then the $\ell$-parity conjecture holds for $E/K$.
\end{prop}

Remark that there is a large supply of quadratic extensions $F/K$ by Theorem~11.2 in~\cite{ulmer_gnv}.
The main problem here seems to find the extension $L_a/F_{an}$. Theorem~5.2 in~\cite{ulmer_gnv} provides us with many extensions that satisfy all the properties except that we can not guarantee that $L_a/K_{an}$ is dihedral. We first had hopes that Ulmer's proof could be adapted to enforce that $L/K_{n}$ is dihedral. In the notations of~\cite{ulmer_gnv}, we may sketch the problem. Let $D$ be a divisor of large degree as in section 6.2. Then the density (as $n$ grows) of elements in the Riemann-Roch space $H^1(\mathcal{C}\times \FF_{q^n}, \mathcal{O}(D))$ which give rise to a dihedral extension of $F_n$ with respect to the fixed quadratic extension $F/K$ will be tending very fast to 0. So we would need to modify the parameter space $X$ and it is not clear how to find a nice variety parametrising such dihedral extensions.

\begin{proof}
 This is very similar to the proof of Theorem~\ref{ell_thm}. By Proposition~\ref{red_prop}, we may assume that $a=1$ and $n=1$. So $L/K$ is a dihedral extension with group $G$. Let $\rho$ be the irreducible $\QQ_{\ell}[G]$-module of degree $\ell-1$ and let $\varepsilon$ the character corresponding to the quadratic extension $F/K$. The usual relation of induced representation $\Theta$ for $G$, as in Example~3.5 in~\cite{dok_sd}, yields the congruence
 \begin{equation*}
  m_{\mathbbm{1}} + m_{\varepsilon} + m_{\rho} \equiv \ord_{\ell}\Bigl(\frac{C(E/L)}{C(E/F)}\Bigr) \pmod{2}.
 \end{equation*}
 We have $m_{\rho} = 0$ and from Lemma~\ref{fracc_lem}, we know that the assumption that no bad place ramifies in $L/K$ implies that $\frac{C(E/L)}{C(E/F)}$ has even $\ell$-adic valuation. This implies that $m_{\mathbbm{1}} + m_{\varepsilon}$ is even, i.e. the $\ell$-parity is valid for $E/F$. By Proposition~\ref{red_prop} implies that the $\ell$-parity also holds for $E/K$.
\end{proof}

\section*{Acknowledgements}
 We express our gratitude to Tim and Vladimir Dokchitser, Jean Gillibert, Christian Liedtke, James S. Milne, Takeshi Saito,  Ki-Seng Tan, Douglas Ulmer, and the anonymous referee for useful comments on the preliminary versions of this paper.

\bibliographystyle{amsalpha}
\bibliography{parity}

\providecommand{\bysame}{\leavevmode\hbox to3em{\hrulefill}\thinspace}
\providecommand{\MR}{\relax\ifhmode\unskip\space\fi MR }
\providecommand{\MRhref}[2]{%
  \href{http://www.ams.org/mathscinet-getitem?mr=#1}{#2}
}
\providecommand{\href}[2]{#2}
\begin{thebibliography}{CFKS10}

\bibitem[Cas65]{cassels}
J.~W.~S. Cassels, \emph{Arithmetic on curves of genus 1. {VIII}. {O}n
  conjectures of {B}irch and {S}winnerton-{D}yer}, J. Reine Angew. Math.
  \textbf{217} (1965), 180--199.

\bibitem[CFKS10]{cfks}
John Coates, Takako Fukaya, Kazuya Kato, and Ramdorai Sujatha, \emph{Root
  numbers, {S}elmer groups, and non-commutative {I}wasawa theory}, J. Algebraic
  Geom. \textbf{19} (2010), no.~1, 19--97.

\bibitem[DD08]{dok_isogeny}
Tim Dokchitser and Vladimir Dokchitser, \emph{Parity of ranks for elliptic
  curves with a cyclic isogeny}, J. Number Theory \textbf{128} (2008), no.~3,
  662--679.

\bibitem[DD09a]{dok_reg}
\bysame, \emph{Regulator constants and the parity conjecture}, Invent. Math.
  \textbf{178} (2009), no.~1, 23--71.

\bibitem[DD09b]{dok_09}
\bysame, \emph{Root {Numbers} and {Parity} of {Ranks} for {Elliptic} {Curves}},
  preprint available at {\url{http://arxiv.org/abs/0906.1815}}, 2009.

\bibitem[DD09c]{dok_sd}
\bysame, \emph{Self-duality of {S}elmer groups}, Math. Proc. Cambridge Philos.
  Soc. \textbf{146} (2009), no.~2, 257--267.

\bibitem[DD10]{dok_modsquares}
\bysame, \emph{On the {B}irch-{S}winnerton-{D}yer quotients modulo squares},
  Ann. of Math. (2) \textbf{172} (2010), no.~1, 567--596.

\bibitem[Dok05]{dok_nonab}
Vladimir Dokchitser, \emph{Root numbers of non-abelian twists of elliptic
  curves}, Proc. London Math. Soc. (3) \textbf{91} (2005), no.~2, 300--324,
  With an appendix by Tom Fisher.

\bibitem[GA09]{Go}
Cristian~D. Gonz{\'a}lez-Avil{\'e}s, \emph{Arithmetic duality theorems for
  1-motives over function fields}, J. Reine Angew. Math. \textbf{632} (2009),
  203--231.

\bibitem[GAT07]{GT}
Cristian~D. Gonz{\'a}lez-Avil{\'e}s and Ki-Seng Tan, \emph{A generalization of
  the {C}assels-{T}ate dual exact sequence}, Math. Res. Lett. \textbf{14}
  (2007), no.~2, 295--302.

\bibitem[Hel09]{helfgott}
Harald~Andres Helfgott, \emph{On the behaviour of root numbers in families of
  elliptic curves}, preprint available at
  \url{http://arxiv.org/abs/math.NT/0408141}, 2009.

\bibitem[HS09]{HS}
David Harari and Tam{\'a}s Szamuely, \emph{Corrigenda for: {A}rithmetic duality
  theorems for 1-motives}, J. Reine Angew. Math. \textbf{632} (2009), 233--236.

\bibitem[Kim07]{kim}
Byoung~Du Kim, \emph{The parity conjecture for elliptic curves at supersingular
  reduction primes}, Compos. Math. \textbf{143} (2007), no.~1, 47--72.

\bibitem[KM85]{katz_mazur}
Nicholas~M. Katz and Barry Mazur, \emph{Arithmetic moduli of elliptic curves},
  Annals of Mathematics Studies, vol. 108, Princeton University Press,
  Princeton, NJ, 1985.

\bibitem[KT03]{kato_trihan}
Kazuya Kato and Fabien Trihan, \emph{On the conjectures of {B}irch and
  {S}winnerton-{D}yer in characteristic {$p>0$}}, Invent. Math. \textbf{153}
  (2003), no.~3, 537--592.

\bibitem[MD70]{SGA3}
A.~Grothendieck M.~Demazure (ed.), \emph{Sch\'emas en groupes. {I}:
  {P}ropri\'et\'es g\'en\'erales des sch\'emas en groupes}, S\'eminaire de
  G\'eom\'etrie Alg\'ebrique du Bois Marie 1962/64 (SGA 3). Dirig\'e par M.
  Demazure et A. Grothendieck. Lecture Notes in Mathematics, Vol. 151,
  Springer-Verlag, Berlin, 1970.

\bibitem[Mil68]{milne_isotrivial}
James~S. Milne, \emph{The {T}ate-\v {S}afarevi\v c group of a constant abelian
  variety}, Invent. Math. \textbf{6} (1968), 91--105.

\bibitem[Mil80]{Milne80}
\bysame, \emph{\'{E}tale cohomology}, Princeton Mathematical Series, vol.~33,
  Princeton University Press, Princeton, N.J., 1980.

\bibitem[Mil06]{milne}
\bysame, \emph{Arithmetic duality theorems}, second ed., BookSurge, LLC,
  Charleston, SC, 2006.

\bibitem[MR07]{mazur_rubin}
Barry Mazur and Karl Rubin, \emph{Finding large {S}elmer rank via an arithmetic
  theory of local constants}, Ann. of Math. (2) \textbf{166} (2007), no.~2,
  579--612.

\bibitem[Nek01]{nekovar2}
Jan Nekov{\'a}{\v{r}}, \emph{On the parity of ranks of {S}elmer groups. {II}},
  C. R. Acad. Sci. Paris S\'er. I Math. \textbf{332} (2001), no.~2, 99--104.

\bibitem[Nek09]{nekovar3}
\bysame, \emph{On the parity of ranks of {S}elmer groups. {IV}}, Compos. Math.
  \textbf{145} (2009), no.~6, 1351--1359, With an appendix by Jean-Pierre
  Wintenberger.

\bibitem[Nek10]{nekovar4}
\bysame, \emph{Some consequences of a formula of {Mazur} and {Rubin} for
  arithmetic local constants}, preprint, 2010.

\bibitem[Roh94]{rohrlich_wd}
David~E. Rohrlich, \emph{Elliptic curves and the {W}eil-{D}eligne group},
  Elliptic curves and related topics, CRM Proc. Lecture Notes, vol.~4, Amer.
  Math. Soc., Providence, RI, 1994, pp.~125--157.

\bibitem[Tat95]{tate}
John Tate, \emph{On the conjectures of {B}irch and {S}winnerton-{D}yer and a
  geometric analog}, S\'eminaire Bourbaki, Vol.\ 9, Soc. Math. France, 1995,
  pp.~Exp.\ No.\ 306, 415--440.

\bibitem[TO70]{oort_tate}
John Tate and Frans Oort, \emph{Group schemes of prime order}, Ann. Sci.
  \'Ecole Norm. Sup. (4) \textbf{3} (1970), 1--21.

\bibitem[Ulm91]{ulmer_pdescent}
Douglas~L. Ulmer, \emph{{$p$}-descent in characteristic {$p$}}, Duke Math. J.
  \textbf{62} (1991), no.~2, 237--265.

\bibitem[Ulm04]{ulmer_analogies}
Douglas Ulmer, \emph{Elliptic curves and analogies between number fields and
  function fields}, Heegner points and {R}ankin {$L$}-series, Math. Sci. Res.
  Inst. Publ., vol.~49, Cambridge Univ. Press, Cambridge, 2004, pp.~285--315.

\bibitem[Ulm05]{ulmer_gnv}
Douglas~L. Ulmer, \emph{Geometric non-vanishing}, Invent. Math. \textbf{159}
  (2005), no.~1, 133--186.

\end{thebibliography}

\end{document}